\documentclass[11pt]{article}

\usepackage{latexsym,color,graphicx,amsmath,amssymb,amsthm,upquote,commath,enumerate}
\usepackage[hidelinks]{hyperref}
\def\g{{\mathfrak{g}}}
\def\a{{\mathfrak{a}}}
\def\R{{\mathbb R}}

\def\Q{{\mathbb Q}}
\def\N{{\mathbb N}}

\def\C{{\mathbb C}}
\def\H{{\mathbb H}}

\def\A{{\mathcal A}}
\def\O{{\mathcal O}}
\def\B{{\mathcal B}}

\def\E{{\mathcal E}}
\def\T{{\mathbb{T}}}
\def\P{{\mathcal P}}

\def\d{{\bf d}}
\def\h{{\bf ht}}

\def\Z{{\mathbb Z}}

\def\M{\text{M}}
\def\vol{\text{vol}}

\def\Hom{\text{Hom}}
\def\SL{\text{SL}}
\def\SO{\text{SO}}
\def\GL{\text{GL}}

\def\PSL{\text{PSL}}

\def\X{\Gamma\backslash G}
\def\Y{\Gamma\backslash G/K}

\newtheorem{thm}{Theorem}[section]

\newtheorem{cor}[thm]{Corollary}
\newtheorem{lem}[thm]{Lemma}

\newtheorem{prop}[thm]{Proposition}
\theoremstyle{definition}
\newtheorem{defin}[thm]{Definition}

\begin{document}
\title{Positive entropy using Hecke operators at a single place}
\date{}
\author{Zvi Shem-Tov\thanks{%
Mathematics Department, University of British Columbia,
Vancouver BC, Canada.
{\sl zvishem@math.ubc.ca} }
 }

\maketitle

\begin{abstract}
We prove the following statement:
Let $X=\SL_n(\Z)\backslash \SL_n(\R)$,
and consider the standard action of the diagonal group 
$A<\SL_n(\R)$ on it. Let $\mu$ be an $A$-invariant 
probability measure on $X$, which is a limit 
$$
\mu=\lambda\lim_i|\phi_i|^2dx,
$$ 
where $\phi_i$ are normalized eigenfunctions of the Hecke algebra at some fixed place $p$, and $\lambda>0$ is some 
positive constant. 
Then any regular element $a\in A$ 
acts on $\mu$ with positive entropy on almost every ergodic component. We also prove a similar result 
for lattices coming from division algebras over $\Q$, and derive a quantum unique ergodicity 
result for the associated locally symmetric spaces. 
This generalizes a result of Brooks and Lindenstrauss \cite{BL}.
\end{abstract}

\section{Introduction}
Let $Y$ be a compact manifold of negative sectional curvature, and $\phi_i$ be a sequence of
normalized eigenfunctions of the Laplacian with eigenvalues $\lambda_i\to \infty$.   
The Quantum Unique Ergodicity conjecture of Rudnick and 
Sarnak \cite{RS} asserts that the only weak-$*$ limit 
of the measures $\mu_i=|\phi_i|^2d\vol_Y$ is $d\vol_Y$. 
An important special case of this problem is the case of $Y$ a 
compact quotient of the upper half plane with its usual hyperbolic metric. In fact, if $G$ is a simple Lie group, $K<G$ is a maximal compact, and $\Gamma<G$ is any lattice, 
then QUE is conjectured to hold true for the locally symmetric space $Y=\Y$ (see e.g. \cite[Problem~6.1]{SV}). 
In this case the $\phi_i$ are assumed to be a sequence of
normalized eigenfunctions of the ring of $G$-invariant differential operators on $G/K$,
with the eigenvalues with respect to the Casimir operator tending to $\infty$ in absolute value.
If $\Gamma$ is a congruence lattice, it is natural to consider sequences $\phi_i$ as above which are 
also eigenfunctions of the Hecke algebra of $Y$. In~\cite{Lin1}, Lindenstrauss
made significant progress and established QUE for 
such sequences $\phi_i$ on compact hyperbolic surfaces $\Gamma\backslash\H$.
For non-compact congruence surfaces Lindenstrauss proved the weaker result 
that any weak-$*$ limit of measures $\mu_i$ as above is proportional to the uniform measure. 
The proof of QUE in this case was later completed by Soundararajan, 
who ruled out escape of mass~\cite{Sound}.  
This is known as \emph{Arithmetic} QUE.
Lindenstrauss' proof is based on his deep results on the dynamics of diagonal 
actions on $\X$. 
The link between QUE and dynamics on $X=\X$ follows from the following general 
idea. To each of the measures $\mu_i$ on $Y$, one associates 
a distribution $\tilde{\mu}_i$ (``microlocal lift'') on the unit cotangent bundle $S^*Y$, which on one hand
projects to $\mu_i$ on $Y$, 
and on the other hand, any weak-$*$ limit of the $\tilde{\mu}_i$ is a probability measure invariant under
the geodesic flow on $S^*Y$.
This construction is originally due Shnirelman, Zelditch and Colin de Verdiere (\cite{Shni}, \cite{Zel}, \cite{Col}). 
In the case of a hyperbolic surface $Y=\Gamma\backslash \H$, 
the unit cotangent bundle is isomorphic
to $X=\Gamma\backslash\PSL_2(\R)$, and under this isomorphism the 
geodesic flow is given by multiplying by a diagonal matrix
on the right. 
Thus, using an appropriate version of the microlocal lift (compatible with the Hecke operators), 
the problem can be 
reduced to the following problem concerning invariant measures on $X$: 
\emph{Let $\phi_i$ be a sequence 
of normalized Hecke eigenfunctions on $X$. Suppose that 
$\mu$ is a weak-$*$ limit of the measures $\mu_i=|\phi_i|^2dx$ on $X$, 
invariant under the diagonal action on $X$. Show that $\mu=dx$.}    
For more details on this reduction and on the microlocal lift we refer to \cite{SV} and \cite{Sil}.
Brooks and Lindenstrauss proved this statement assuming that the $\phi_i$ are eigenfunctions
of only one Hecke operator, for certain co-compact lattices in $\SL_2(\R)$, coming from a quaternion algebra over $\Q$~\cite{BL}.
Their main innovation is showing that for any such $\mu$ 
the diagonal group acts with positive entropy on almost every ergodic component.
This generalizes the original positive entropy result of Bourgain and Lindenstrauss
\cite{BrL}, which uses the full Hecke algebra.
Brooks and Lindenstrauss suggested that their results might be 
generalized to the cases considered 
by Silberman and Venkatesh in their work on QUE in higher rank \cite{SV}.  
In this paper we generalize the techniques of Brooks and Lindenstrauss to higher rank. 
Let $G=\SL_n(\R)$, $K=\SO(n)$, $\Gamma=\SL_n(\Z)$, $A<G$ the diagonal group,
 and fix a prime number $p$. The following theorem is our main result. 
\begin{thm}\label{main-main}
Let $\mu$ be an $A$-invariant 
probability measure on $X$, which is a weak-$*$ limit 
$$
\mu=\lambda\lim_i|\phi_i|^2dx,
$$ 
where $\phi_i$ are normalized eigenfunctions of the Hecke algebra at some fixed place $p$, and $\lambda>0$ is some 
positive constant. 
Then any regular element $a\in A$ 
acts on $\mu$ with positive entropy on almost every ergodic component.  \end{thm}
Using the higher rank microlocal lift constructed by 
Silberman and Venkatesh \cite[Theorem 1.6]{SV2}, and the measure rigidity results of 
Einsiedler, Katok, and Lindenstrauss \cite[Corollary 1.4]{EAL}, we conclude the following result. 
\begin{thm}\label{que}
Assume $n$ is prime. 
Let $\phi_i\in L^2(\Y)$ be a non-degenerate (in the sense of \cite{SV2}) sequence of joint eigenfunctions 
of the ring of $G$-invariant differential operators on $G/K$ and the Hecke operators at some fixed place $p$. 
Then any weak-$*$ limit of the  measures $|\phi_i|^2dy$ is 
proportional to the uniform measure $dy$ on $\Y$.    
\end{thm}
We note that this is a version of the main theorem of \cite{SV}, except with the weaker hypotheses that the $\phi_i$ are eigenfunctions of the Hecke algebra at only one place.   
In fact, after some adjustments to our proof of Theorem \ref{main-main}, the result holds 
also when replacing the lattice $\SL_n(\Z)$ by lattices coming from a division algebra 
of prime degree over $\Q$, and the 
regular element $a$ by any non-trivial element of $A$ (see Theorem \ref{main-main-csa}). 
Since these lattices are co-compact the constant $\lambda$ in Theorem $\ref{main-main}$ must be $1$. 
Using Einsiedler and Katok's theorem \cite[Theorem 4.1]{EK} 
we obtain the following result.

\begin{thm}\label{main-main2}
Let $D$ be a division algebra of prime degree $n$ over $\Q$, which splits over $\R$ and
over $\Q_p$, and
$\O\subset D$ a maximal order. Let $G=\SL_n(\R)$ and $\Gamma<G$ the lattice of all norm $1$ elements
of $\O$.
Let $\phi_i$ be a sequence of normalized eigenfunctions of the Hecke algebra at $p$ of $\X$, and suppose that 
$\mu$ is an $A$-invariant probability measure on $\X$ which is a weak-$*$ limit, 
$$
\mu=\lim_i|\phi_i|^2dx.
$$  
Then $\mu=dx$, the normalized Haar measure on $\X$. 
\end{thm}

In the notation of Theorem \ref{main-main2}, consider the locally symmetric space $Y=\Y$, where $K$ is the maximal 
compact subgroup $K=\SO(n)$. 
Using the higher rank microlocal lift constructed by Silberman and Venkatesh \cite[Theorem 1.6]{SV2}, we conclude
the following QUE result for lattices coming from division algebras of prime degree. 
\begin{thm}\label{que2}
In the notation above, let $\phi_i\in L^2(Y)$ be a non-degenerate (in the sense of \cite{SV2}) sequence of joint eigenfunctions 
of the ring of $G$-invariant differential operators on $G/K$ and the Hecke operators at some fixed place $p$. 
Then the only weak-$*$ limit of $\phi_i$ is the standard uniform probability measure $dy$ on $Y$. 
\end{thm}
   
Note that in Theorem \ref{que} we have not ruled out escape of mass. In fact, for sequences of eigenfunctions of 
the Hecke operators at a single place, escape of mass has not been ruled out 
even in the case where $n=2$ (see \cite{BL}). 
\paragraph{Acknowledgements.} 

I am very grateful to my advisor Lior Silberman for introducing this problem to me. 
I thank Elon Lindenstrauss and Lior Silberman for very helpful discussions, suggestions, and comments.
This work was partially supported by the ERC grant HomDyn no. 833423.

\section{Background on spherical functions on $p$-adic groups}
In this section we review what we need from the theory of spherical functions on $p$-adic groups. 
We mostly follow \cite{M}. Another standard reference is \cite{Sa}.
\paragraph{Notation.}
For two real-valued functions $f,g$ with the same domain, we write $f\ll g$ if there is a positive constant $C$ so that $f\le Cg$. 
We say that $f$ and $g$ are \emph{equivalent} if $f\ll g$ and $g\ll f$. We write
$f\ll_t g$ if the implied constant $C$ above depends on a parameter $t$. 
Fix a prime number $p$. Let $G=\SL_n(\R)$, $G_p=\SL_n(\Q_p)$, and $K_p=\SL_n(\Z_p)$. 
We denote by $l=n-1$ the rank of $G$. Throughout, unless otherwise specified, 
our constants may depend on $G, G_p, K_p$ but not on anything else. 
    
Recall that $K_p$ is a compact open subgroup of $G_p$ (simply because $\Z_p$ is compact and
open in $\Q_p$). 
In particular we can fix a Haar measure on $G_p$ normalized so that its restriction
to $K_p$ is a probability measure. We recall also that $K_p$ is in fact a \emph{maximal} compact
subgroup of $G_p$. 
Let $L(G_p,K_p)$ be the $\C$-algebra of all $K_p$ bi-invariant  
compactly supported functions on $G_p$, the product in $L(G_p,K_p)$ being defined by the 
convolution,
$$
f*g(x)=\int_{G_p} f(xy)g(y^{-1})dy.
$$
Note that any function in $L(G_p,K_p)$ (or more generally, any $K_p$ right-invariant function on $G_p$) 
is automatically continuous, since $K_p$
is open in $G_p$. One can view $K_p$ bi-invariant functions as functions on the quotient $G_p/K_p$ which 
are invariant under the action of $K_p$ on the left. Note here the analogy with the case of 
a symmetric space $G/K$, where $G$ is a semisimple Lie group and 
$K$ a compact subgroup.

A complex valued function $\omega$ on $G_p$ is called a \emph{zonal spherical function} 
(or just spherical function) on $G_p$ relative to $K_p$ if it is $K_p$-bi-invariant, $\omega(1)=1$, and 
for all $f\in L(G_p,K_p)$, $\omega$ is an eigenfunction of the integral operator defined by $f$, i.e. we have
\begin{equation}\label{sf}
f*\omega=\lambda_f\omega,
\end{equation}
with $\lambda_f\in\C$. For a spherical function $\omega$ we denote the scalar $\lambda_f$ 
in \eqref{sf} by 
$\hat\omega(f)$, i.e. we put 
$$
\hat\omega(f)=\int_{G_p}f(g)\omega(g^{-1})dg.
$$
Then clearly $\hat\omega$ is a $\C$-algebra homomorphism from $L(G_p,K_p)$ onto $\C$. 
Moreover, the correspondence $\omega\mapsto\hat{\omega}$ 
is a bijection between spherical functions 
on $G_p$ relative to $K_p$ and non-zero $\C$-algebra homomorphisms $L(G_p,K_p)\to \C$.
Denote by $\Omega$ the set of all spherical functions on $G_p$ relative to $K_p$. 
For each $f\in L(G_p,K_p)$ its \emph{spherical transform} $\hat{f}:\Omega\to \C$ is defined by 
$\hat{f}(\omega)=\hat\omega(f)$. 
We say that a spherical function $\omega$ is \emph{positive definite} if $\omega(g_ig_j^{-1})$ is a positive
definite matrix for any $g_1,\dots,g_m\in G_p$, and we denote by $\Omega^+$ the set of all positive definite spherical
functions on $G_p$ relative to $K_p$. Let $L^2(G_p,K_p)$ denote the space of square-integrable functions 
on $G_p$ which are 
bi-invariant with respect to $K_p$.  
\begin{thm}[\bf{\cite[Theorem~(1.5.1)]{M}}]\label{thm3}
There is a unique positive measure $\mu$ on $\Omega^+$ such that 
\begin{enumerate}
\item If $f\in L(G_p,K_p)$ then $\hat{f}\in L^2(\Omega^+,\mu)$,
\item $\int_{G_p} |f(x)|^2dx=\int_{\Omega^+}|\hat{f}|^2d\mu(\omega)$ for all $f\in L(G_p,K_p)$.
\end{enumerate}
Moreover, the mapping $f\mapsto \hat{f}$ extends to an isomorphism of Hilbert spaces 
$L^2(G_p,K_p)\to L^2(\Omega,\mu)$.
\end{thm}
The measure $\mu$ is called the \emph{Plancherel measure} on $\Omega^+$.

The algebra $L(G_p,K_p)$, 
as well as the spherical functions on $G_p$ and the Plancherel measure, can be described in terms of an
affine root structure on $G_p$, which we describe below.   

Let $V$ be the space of all vectors $v=(v_0,\dots,v_l)\in\R^{l+1}$ 
with $\sum_{i=0}^lv_i=0$, endowed with the inner product 
$\left<u,v\right>=\sum_iu_iv_i$.   
Let $V^*$ be the dual of $V$. Then $V$ and $V^*$ are naturally isomorphic through 
$\left<\cdot,\cdot\right>$.
For each non-zero $a\in V^*$,  
let $a^{\vee}\in V$ be the image of $\frac{2a}{\left<a,a\right>}$ under this isomorphism. 
Denote by $e_i$ the $i$th coordinate function on $V$, and consider the root system 
\begin{equation}
\Sigma_0=\{e_i-e_j\mid i\ne j\}\subset V^*,
\end{equation}
as well as a fixed set of simple roots 
\begin{equation}\label{simple}
\Pi_0=\{a_i\mid 1\le i\le l\},
\end{equation}
where $a_i:=e_{i-1}-e_i$.
For each $a\in\Sigma_0$ and $k\in\Z$ let $a+k$ be the affine function on $V$
given by $v\mapsto a(v)+k$. Let 
\begin{equation}
\Sigma=\{a+k\mid a\in\Sigma_0, k\in\Z\}.
\end{equation}
The elements of $\Sigma$ are called \emph{affine roots}. 
For each $\alpha\in \Sigma$ let $w_\alpha$ be the orthogonal reflection in the (affine) hyperplane on which 
$\alpha$ vanishes. 
The group $W$ generated by the $w_\alpha$ is an infinite group of affine transformations from $V$ to itself, 
called the \emph{affine Weyl group} of $\Sigma$. 
Let $W_0$ be the subgroup of $W$ which fixes the point $0$; this is the Weyl group of $\Sigma_0$. 
The set of all translations in $W$ is a free abelian group $T$ of rank $l$, 
and we have that $W$ is the semi-direct product 
\begin{equation}\label{sdprod}
W=W_0\ltimes T
\end{equation}
of $T$ and $W_0$.  
For each $a\in\Sigma_0$, let 
\begin{equation}\label{eq: t def}
    t_a=w_a\circ w_{a+1}\in T. 
\end{equation}
The $t_a$ ($a\in\Pi_0$) are a basis of $T$. We have $t_a(0)=a^{\vee}$ and the mapping $T\to V$ defined by 
$t\mapsto t(0)$ maps $T$ isomorphically onto the lattice $L$ spanned by $a^{\vee}$ ($a\in\Sigma_0)$.  
It is easy to see that 
\begin{equation}
L=\{(n_0,\dots,n_l)\in\Z^{l+1}\mid\sum_in_i=0\}.
\end{equation}
Let $E_i\in \R^{l+1}$ ($i=0,\dots,l$) be the standard basis. 
We obtain another system of coordinates for $T$ by sending the basis 
$E_{i-1}-E_i$ ($i=1,\dots,l$) of $L$ to the standard basis of $\Z^l$.  
Namely, $(n_0,\dots,n_l)\mapsto (m_1,\dots,m_l)$, where

\begin{equation}\label{coords}
\begin{split}
(n_0,\dots,n_l)&=(m_1, m_2-m_1, m_3-m_2,\dots,m_l-m_{l-1},-m_l)\\
(m_1,\dots,m_l)&=(n_0, n_0+n_1,\dots, n_0+\dots+n_{l-1}).
\end{split}
\end{equation}
Throughout we use both the coordinates $(n_0,\dots,n_l)$ 
and the coordinates $(m_1,\dots,m_l)$ for $T$ interchangeably.

Let $Z$ be the group of all diagonal matrices in $G_p$. We have a surjective homomorphism 
\begin{equation}
\nu:Z\to T,
\end{equation}
which maps $\text{diag}(\lambda_0,\dots,\lambda_l)$ to the translation by 
the vector $\sum_{i=0}^lv_p(\lambda_i)E_i\in V$. Here $v_p:\Q_p\to \Z$ is the standard $p$-adic valuation.
Let 
\begin{equation}
S=\Hom(T,\C^{\times}).
\end{equation}
Viewing $T$ as a quotient of $Z$ one constructs for each $s\in S$
a spherical function $\omega_s$, using induction of characters. See {\cite[(3.3)]{M}} for 
the details and the precise definition of $\omega_s$. 
\begin{thm}[\bf{\cite[Theorem~3.3.12]{M}}]\label{thm1}
Every spherical function is of the form $\omega_s$ for some $s\in S$. Furthermore, 
$\omega_s=\omega_{s'}$ if and only if $s=ws'$ for some $w\in W_0$. 
\end{thm}

Let $L(T)$ be the space of compactly supported functions on $T$. 
The action of $W_0$ on $T$ by conjugation (see \eqref{sdprod}) induces an action of $W_0$
on $L(T)$: $(w\phi)(t)=\phi(w^{-1}t w)$ ($\phi\in L(T), t\in T, w\in W_0$). 
Let $L(T)^{W_0}$ be the space of finitely supported, $W_0$-invariant functions on $T$. 
\begin{thm}[\bf{Satake Isomorphism, {\cite[Theorem~3.3.6]{M}}}]\label{thm2}
Let $f\in L(G_p,K_p)$. There exists a unique function $\tilde{f}\in L(T)^{W_0}$ such that for any $s\in S$, 
$$
\hat{\omega}_s(f)=\sum_{t\in T}\tilde{f}(t)s(t).
$$
The map $f\mapsto \tilde{f}$ is an isomorphism of $\C$-algebras from $L(G_p,K_p)$ to $L(T)^{W_0}$.
\end{thm}
The $\tilde{f}(t)$ in Theorem \ref{thm2} should be regarded as the Fourier coefficients of $f$. 
Let $U^-<G_p$ be the group of all lower diagonal matrices 
with $1$'s on the diagonal. Then $U^-$ is nilpotent, and so it is unimodular. The diagonal group $Z$ normalizes 
$U^-$ and we have the following formula for the Jacobian of the action of $Z$ on $U^-$ 
by conjugation {\cite[Proposition~(3.2.4)]{M}}.
Let $\Sigma_0^+$ be the set of positive roots of $\Sigma_0$ (i.e. $e_i-e_j$ with $i<j$) 
and $r=\frac{1}{2}\sum_{a\in\Sigma_0^+}a$. For any $t\in T$ and $z\in\nu^{-1}(t)$,
let $\Delta(z)=\delta(t)=p^{2r(t(0))}$. Then  
$$
\frac{dzu^-z^{-1}}{du^-}=\Delta(z)^{-1}. 
$$
In the notation above we have the following formula ({\cite[(3.3.4)]{M}}) for $\tilde{f}$. 
\begin{equation}\label{ftilde}
\tilde{f}=\delta^{-\frac{1}{2}}(t)\int_{U^-}f(zu^-)du^-.
\end{equation}

\begin{defin}
For $f\in L(G_p,K_p)$ define the \emph{adjoint} of $f$ to be
$f^*(g)=\overline{f(g^{-1})}$.
\end{defin}
\begin{lem}
For any $f\in L(G_p,K_p)$, $\tilde{f}(t)=\overline{\tilde{(f^*)}(-t)}$. 
\end{lem}
\begin{proof}
We have by \eqref{ftilde} that
\begin{align*}
\tilde{(f^*)}(t)=&\delta^{-\frac{1}{2}}(t)\int_{U^-}f^*(zu^-)du^-\\
=&\delta^{-\frac{1}{2}}(t)\int_{U^-}\overline{f((u^-)^{-1}z^{-1})}du^-\\
=&\overline{\delta^{-\frac{1}{2}}(t)\int_{U^-}f((u^-)^{-1}z^{-1})du^-}.
\end{align*}
Using the unimodularity of $U^-$ this yields 
\begin{align*}
\int_{U^-}f((u^-)^{-1}z^{-1})du^-=& \int_{U^-}f(z^{-1}zu^-z^{-1})du^-\\
=&\Delta(z)\int_{U^-}f(z^{-1}u^-)du^-, 
\end{align*}
and so 
$$
\overline{\tilde{(f^*)}(t)}=\delta^{\frac{1}{2}}(t)\int_{U^-}f(z^{-1}u^-)du^-
=\delta^{-\frac{1}{2}}(-t)\int_{U^-}f(z^{-1}u^-)du^-=\tilde{f}(-t),
$$
as needed. 
\end{proof}

We denote by $\hat{T}$ the set of all $s\in S$ with $|s(t)|=1$ $\forall t\in T$. 
We say that an element $s\in S$ is \emph{nonsingular} if 
\begin{equation}\label{nonsing}
s(t_a)\ne1,
\end{equation} 
for all $a\in\Sigma_0$. We call an element $s$ in $S$ \emph{singular} if it is not nonsingular.  
For $s\in \hat{T}$, write $s(t_{e_i-e_j})=\xi_i\xi_j^{-1}$ ($\xi_i\in\C$). If $s$ is nonsingular, let 
\begin{equation}\label{hc}
c(s)=\prod_{i<j}\frac{\xi_i-p^{-1}\xi_j}{\xi_i-\xi_j},
\end{equation}
the \emph{Harish-Chandra function}, where $t_{e_i-e_j}$ is as in \eqref{eq: t def}. 
Notice that for $s\in\hat{T}$, we must have $|\xi_i|=|\xi_j|$ for all $i<j$, so that 
the numerator in \eqref{hc} never vanishes. Thus $c(s)^{-1}$ is well-defined  
for all $s\in\hat{T}$. 
Notice that $\hat{T}$ can be identified with the $l$-dimensional torus $\R^l/\Z^l$, by sending 
$$
s\in\hat{T}\mapsto (\theta_k\Z)_{k=1}^l,
$$
if $s(t_{a_k})=e^{2\pi i\theta_k}$ for all $1\le k\le l$ (see \eqref{simple}). 
In terms of this identification $s\mapsto c(s)^{-1}$ is real
analytic, due to \eqref{hc}. 
We will also use the fact that $|c(s)|^{-2}$ is $W_0$-invariant, hence
\begin{equation}\label{inv}
|c(s)|^{-2}=(c(ws)c(ws^{-1}))^{-1},
\end{equation}
for all $w\in W_0$ ({\bf{\cite[~(5.1.4)]{M}}}).
\begin{thm}[\bf{\cite[Theorem~(5.1.2)]{M}}]\label{planch}
The Plancherel measure $\mu$ on $\Omega^+$ is concentrated on the set $\{\omega_s\mid s\in\hat{T}\}$,
and up to normalization is given by 
$$
d\mu(\omega_s)=\frac{ds}{|c(s)|^2}.
$$
Here $ds$ is the probability Haar measure on the compact group $\hat{T}$. 
\end{thm}
By convention we write integrals over $\Omega^+$ with respect to the Plancherel measure as follows:
if $h$ is a function on $\Omega^+$ we write
$$
\int_{\hat{T}}h(s)d\mu(s),
$$
meaning 
$$
\int_{\hat{T}}h(s)\frac{ds}{|c(s)|^2}.
$$

We denote by $Z^{++}$ the elements $\text{diag}(\lambda_0,\dots,\lambda_l)\in Z$ 
such that 
$$
v_p(\lambda_0)\ge v_p(\lambda_1)\ge\dots\ge v_p(\lambda_l),
$$ 
and its image under $\nu$ by $T^{++}\subset T$.
\begin{thm}[\bf{Cartan decomposition, {\cite[Theorem~(2.6.11)]{M}}}]
$G_p=K_pZ^{++}K_p$, and the mapping $t\mapsto K_p\nu^{-1}(t)K_p$ is a bijection of 
$T^{++}$ onto $K_p\backslash G_p/K_p$. 
\end{thm}
We recall Macdonald's formula for spherical functions. 
We first treat the case of nonsingular elements of $S$ (see \eqref{nonsing}). 
Up to normalization we have for any $z_0\in Z^{++}$ and any nonsingular $s\in S$,  
\begin{equation}\label{mf}
\omega_s(z_0^{-1})=\delta(t_0)^{-\frac{1}{2}}\sum_{w\in W_0}(ws,t_0)c(ws) 
\end{equation}
({\cite[Theorem~(4.2.1)]{M}}).
Here $(ws,t_0)$ simply means $ws(t_0)$, and as before $t_0=\nu(z_0)$.  
In view of the Cartan decomposition the formula above completely determines $\omega_s$.
For a singular $s\in S$ a similar formula can be obtained by taking a limit.

Using the Cartan decomposition we obtain for every 
$g\in G_p$ a representative $t=t(g)\in T^{++}$.
If $t(g)$ is a translation by an element $\sum_{i=0}^l n_iE_i\in L$ we say 
that $p^{-n_l}$ is the \emph{denominator}
of $g$ and we denote it by $\d(g)$. 

\begin{lem}[{\bf{Silberman--Venkatesh}~\cite[Lemma 4.3]{SV}}]\label{lem-lior}
$\d(gg')\le\d(g)\d(g')$ and $\d(g^{-1})\le\d(g)^l$.
\end{lem} 
For every $t\in T$ we denote by $\norm{t}$ the norm of $t(0)$ coming from the inner product on $V$.  
\begin{lem}\label{lem-d}
There are constants $C,C'>0$ such that for any $g\in G_p$
if we let $t=t(g)\in T^{++}$ be the representative of $g$ given by the Cartan decomposition then 
$$
p^{C\norm{t}}\le \d(g)\le p^{C'\norm{t}}.
$$
\end{lem}
\begin{proof}
Since everything is defined by taking representatives in $T^{++}$ it suffices to show that 
there are $C,C'$ such that for every $t=\sum_i n_ie_i$ with $n_0\ge\dots\ge n_l$,
$$
C\norm{t}\le -n_l\le C'\norm{t}. 
$$
Consider the function on $V$ given by $v=\sum_iv_iE_i\in V\mapsto -\min_i\{v_i\}$.
This function is equivalent to the norm on $V$, given by $v\mapsto \max_i\{|v_i|\}$. 
Thus identifying $T^{++}$ with a subset of $V$ we can extend 
$t=\sum_i n_ie_i\in T^{++}\mapsto -\min_i\{n_i\}=-n_l$ to a function which is equivalent to a norm
on $V$. On the other hand $t\in T^{++}\mapsto \norm{t}$ extends to a norm on $V$. Since any two norms 
on $V$ are equivalent, the result follows.   
\end{proof}

For $t\in T$ with $t(0)=\sum_{i=1}^lm_ia_i^{\vee}$ we define the \emph{height} of $t$ to be 
$$
\h(t)=\sum_i|m_i|.
$$
\begin{lem}\label{lem-same}
Let $t_1,\dots,t_m$ be a set of generators of the semigroup $T^{++}$. Then 
the length of $t\in T^{++}$ with respect to $t_1,\dots,t_m$ is equivalent 
to $\h(t)$. That is, there exist constants $C,C'>0$ such that if $t=t_{i_1}\cdots t_{i_k}$
(with minimal $k$) then 
$$
C\h(t)\le k\le C'\h(t).
$$
\end{lem}
\begin{proof}
Let $M>0$ be such that the height of any of the $t_i$'s is less than $M$. 
Then the height of $t=t_{i_1}\cdots t_{i_k}$ is at most $kM$. The other direction 
follows from the fact that the height is additive on $T^{++}$ and the height of any of the 
$t_i$'s is at least $1$.  
\end{proof}

\begin{lem}\label{lem-same2}
The height of $t\in T^{++}$ is equivalent to $m_1$, where as before $t(0)=\sum_{i=1}^lm_ia_i^{\vee}$.
\end{lem}
\begin{proof}
Write $t(0)=\sum_{i=0}^l n_iE_i$ with $n_0\ge n_1\ge\dots\ge n_l$, $\sum_in_i=0$. Then
$$
t(0)=\sum_{i=0}^{l-1}(n_0+\dots+n_i)(E_i-E_{i+1})=\sum_{i=1}^l(n_0+\dots+n_{i-1})a_i^{\vee},
$$  
so that $m_i=n_0+\dots+n_{i-1}$ and $\h(t)=ln_0+(l-1)n_1+\dots+n_{l-1}$. Thus, since $n_0=m_1$, 
$$
m_1\le \h(t)\le (l+(l-1)+\dots+1)m_1.
$$
\end{proof}

\begin{defin} \label{ordering}
Let $t\in T$ and suppose that 
$t(0)=\sum_{i=1}^lm_ia_i^{\vee}$ with $m_i\in\Z$. Define $t\ge0$ if the first non-zero $m_i$ is
positive, and $t_1\ge t_2$ if $t_1t_2^{-1}\ge0$. 
\end{defin}
It is easy to see that $\le$ is a total ordering on $T$. 

\begin{lem}\label{lem-same3}
Let $t,t'\in T^{++}$. Suppose that $t\le t'$ and the lengths of $t,t'$, with respect to a set of generators $t_1,\dots,t_m$
of the semigroup $T^{++}$, are $k,k'$. Then $k\le Ck'$ for some constant $C$ (depending only on the set of generators 
$t_1,\dots,t_m$).   
\end{lem}
\begin{proof}
Suppose that $t(0)=\sum_{i=1}^lm_ia_i^{\vee}$ and $t'(0)=\sum_{i=1}^lm_i'a_i^{\vee}$. Then $m_1\le m_1'$ and
so the result follows from Lemma \ref{lem-same2} and Lemma \ref{lem-same}. 
\end{proof}

\section{The propagation lemma}\label{sec-ker}
\subsection{Statement of result}
We keep the notation of the previous section. 
Let $\Gamma=\SL_n(\Z)$, and $\Gamma_p=\SL_n(\Z[\frac{1}{p}])$.  
Identifying $\Gamma_p$ with its diagonal 
embedding in $G\times G_p$ we have the isomorphism
\begin{equation}\label{decomp}
\X\cong \Gamma_p\backslash G\times G_p/K_p,
\end{equation}
given by $\Gamma x\mapsto \Gamma_p(x,1) K_p$. 
Through this isomorphism $L(G_p,K_p)$ acts on $L^2(\X)$ by
\begin{equation}\label{action}
(\phi*h)(x)=\int_{G_p}\phi(xy)h(y^{-1})dy\quad(\phi\in L^2(\X),h\in L(G_p,K_p)).
\end{equation}
Here $\phi(xy)$ means applying $\phi$ to the preimage of $\Gamma_p(x,y)K_p$
under the isomorphism \eqref{decomp}. 
In this way $L(G_p,K_p)$ is viewed as an algebra of operators 
on $L^2(\X)$, called the \emph{Hecke operators at $p$}. 
Suppose that $\phi\in L^2(\X)$ is an $L(G_p,K_p)$ joint eigenfunction. Then $\phi$ induces a homomorphism 
$h\mapsto \lambda_h$ from $L(G_p,K_p)$ to $\C$, defined by $\phi*h=\lambda_h\phi$. Let $\omega_{s}$
be the corresponding spherical function on $G_p$. We call $s$ the \emph{spectral parameter} of $\phi$.  

In this section, we construct a convolution kernel $K_N\in L(G_p,K_p)$ which spectrally amplifies 
a given $L(G_p,K_p)$-eigenfunction, but still has small norm.
\begin{prop}\label{mainlem}
Let $0<\epsilon<1$. For any sufficiently large $N\in\N$ (depending on $\epsilon$),
and any $L(G_p,K_p)$-eigenfunction $\phi\in L^2(\X)$, there exists a self-adjoint 
kernel $K_N\in L(G_p,K_p)$ satisfying 
\begin{enumerate}

\item $K_N$ is supported on elements with 
denominator at most $p^{rN}$, for some positive constant $r$ 
(depending only on $G$, $G_p$, $K_p$).
\item There exists $\delta>0$, depending only on $\epsilon$, such that 
$$
||K_N||_{\infty}\ll e^{-N\delta}.
$$
\item $\phi$ has $K_N$-eigenvalue $\ge\frac{1}{\epsilon}$.
\item For every $\psi\in L^2(X)$ we have $\left<\psi*K_N,\psi\right>\ge-\norm{\psi}_2^2$.  
\end{enumerate} 
\end{prop}
The rest of this section is devoted to the proof of Proposition \ref{mainlem}.  
\subsection{Construction of $k_N$}
For any $L,q\in\N$, let $g_{L,q}(k)$ be the function on $\Z$ defined by 
$$
\sum_k g_{L,q}(k)e^{kz}=D_L(qz),
$$
where $D_L(z)=\sum_{k=-L}^Le^{kz}$ is the Dirichlet kernel. We will later choose the parameters $L$ and $q$
based on $\epsilon$ and $\phi$ from Proposition \ref{mainlem}. 
Explicitly, if $q\ne0$ 
$$
g_{L,q}(k)=
\begin{cases}
1&\text{if }q|k \text{ and } |k|\le qL,\\
0 &\text{otherwise},
\end{cases}
$$
and if $q=0$ then 
$$
g_{L,0}(k)=
\begin{cases}
2L+1 &\text{if } k=0,\\
0 &\text{otherwise}.
\end{cases}
$$
We extend the definition of $g_{L,q}$ to $l$-tuples. 
If $q$ is a tuple $q=(q_1,\dots,q_l)\in\Z^l$ let
$$
g_{L,q}(m_1,\dots,m_l)=\prod_{i=1}^lg_{L,q_i}(m_i),
$$
and view it as a function on $T\cong \Z^l$ (see \eqref{coords}). 
We define an element $\tilde{f}\in L(T)^{W_0}$ by averaging $g_{L,q}$ over $W_0$, and subtracting 
a multiple of the delta function to have $\tilde{f}(0)=0$. Explicitly, 
\begin{equation}
\tilde{f}_{L,q}(t)=\frac{1}{|W_0|}\sum_{w\in W_0}g_{L,q}(wtw^{-1})-g_{L,q}(0)\delta_0(t).
\end{equation}

\begin{lem}\label{lem141}
For every $L,q\in\N$, the support of $\tilde{f}_{L,q}$ consists of elements $t\in T$ with $\norm{t}\ll \norm{q}_\infty L$.
Also, $\sup_{t\in T}|\tilde{f}_{L,q}(t)|\ll_L1$. 
\end{lem}

\begin{proof}
We claim that there is a positive constant $C$ such that for any $\norm{t}> C\norm{q}_\infty L$ 
and $w\in W_0$, $g_{L,q}(wtw^{-1})=0$. 
Indeed, we start with $w=e$. In this case, $g_{L,q}(t)=\prod_{i=1}^lg_{L,q_i}(m_i)$, 
where $t(0)=\sum_im_ia_i^{\vee}$, and if this does not
vanish it means that $|m_i|\le q_iL$ ($i=1,2,\dots,l$) and so $\norm{(m_1,m_2,\dots,m_l)}_{\infty}\le\norm{q}_{\infty}L$. By equivalence 
of norms $\norm{t}\le C\norm{q}_{\infty}L$ for some $C>0$. This proves the claim for the case where $w=e$. 
Since $W_0$ acts
by isometries on $T$ 
the claim follows for any $w\in W_0$, which proves the first part of the lemma.
The second part is clear from the definition.  
\end{proof}

Fix $L\in\N$, and an $L(G_p,K_p)$-eigenfunction $\phi\in L^2(\X)$ 
with spectral parameter $s\in S$.  
For any $N\in\N$, large enough in terms of $L$, we define $k_N\in L(G_p,K_p)$ 
through its Fourier coefficients $\tilde{k}_N(t)$ as follows.
Select $q_1=q_1(N)\in\N$ such that $q:=(q_1,0,\dots,0)\in \N^l$ has the following properties.

\begin{enumerate}[(i)]\label{properties}
\item $|\hat\omega_s(f_{L,q})|\ge L^l$,\label{A}
\item $N\ll_Lq_1\ll \frac{N}{L}$.\label{B}
\end{enumerate}
The existence of such $q_1$ is established in Lemma \ref{exists} below.
Finally define 
\begin{equation}\label{kerdef}
\tilde{k}_N=\tilde{k}_{L,N}=\tilde{f}_{L,q}*\tilde{f}^*_{L,q}-\tilde{f}_{L,q}*\tilde{f}^*_{L,q}(0)\delta_0,
\end{equation}
where $q=q(N)$ is as above. 

\begin{lem}\label{exists}
There exists $q_1=q_1(N)\in\N$ which satisfies properties \eqref{A} and \eqref{B} above. 
\end{lem}

\begin{proof}
To simplify the notation we 
denote $f=f_{L,q}$, $g=g_{L,q}$, and $g_i=g_{L,q_i}$. We have for $q=(q_1,0,\dots,0)$,  
$$
\hat\omega_s(f)=\sum_{t\in T}\tilde{f}(t)s(t)=
\frac{1}{|W_0|}\sum_{t\in T}\sum_{w\in W_0}g(wtw^{-1})s(t)-(2L+1)^{l-1},
$$
and by reversing the order of summation this is equal to
\begin{equation}\label{eq10}
\frac{1}{|W_0|}\sum_{w\in W_0}\sum_{t\in T}g(t)s(w^{-1}tw)-(2L+1)^{l-1}.
\end{equation}
For each $w\in W_0$, $t\mapsto s(w^{-1}tw)$ is given by  
$$
(m_1,\dots,m_l)\mapsto e^{\sum_i m_iz^{(w)}_i},
$$
for some $z^{(w)}_1,\dots,z^{(w)}_l\in \C$. In this notation, we have
\begin{align*}
\sum_{t\in T}g(t)s(w^{-1}tw)&=\sum_{m_1,\dots,m_l\in\Z}g(m_1,\dots,m_l)e^{\sum_i m_iz^{(w)}_i}\\
&=
\prod_{i=1}^l\sum_{m_i\in\Z} g_i(m_i)e^{m_iz^{(w)}_i}\\
&=
\prod_{i=1}^lD_{L}(q_iz^{(w)}_i).
\end{align*}
Thus 
$$
\hat\omega_s(f)=\frac{1}{|W_0|}\sum_{w\in W_0}\prod_{i=1}^lD_L(q_iz^{(w)}_i)-(2L+1)^{l-1}.
$$
Since $q_i=0$ for $i=2,\dots,l$ we have  
$$
\hat\omega_s(f)=(2L+1)^{l-1}(\frac{1}{|W_0|}\sum_{w\in W_0}D_L(q_1z^{(w)}_1)-1).
$$
Let $z^{(w)}_1=x_w+iy_w$. 
Using a quantitative version of Kronecker
theorem (see Lemma \ref{applem}) we can find $N\ll_L q_1\ll \frac{N}{L}$ such that 
$$
|e^{q_1iky_w}-1|\le\frac{1}{2},
$$
for all $|k|\le L, w\in W_0$.
The lemma now follows since for any such $q_1$ we have 
$$
\frac{1}{|W_0|}|\sum_{w\in W_0}(D_L(q_1z^{(w)}_1)-1)|\ge \frac{L}{2}.
$$ 
\end{proof}
\subsection{Spectral properties of $k_N$}
For the rest of this section, we keep the convention from the proof of Lemma \ref{exists} and denote 
$f=f_{L,q}$, $g=g_{L,q}$, and $g_i=g_{L,q_i}$ (here $q=(q_1,0,\dots,0)$ is our fixed choice from \eqref{kerdef}).
\begin{lem}\label{suppsize}
The support of $\tilde{f}$ is of size at most $O(L)$.
\end{lem}
\begin{proof}
Let $t\in T$, and assume that $\tilde{f}(t)\ne0$. In particular $t\ne0$. 
We claim that there exists a non-zero multiple of $q_1$, say $y\in \Z$, such that $t$ belongs to 
the orbit of $(y,0,\dots,0)$ under $W_0$. 
Indeed, since $\tilde{f}(t)\ne0$, $g(wtw^{-1})\ne0$ for
some $w\in W_0$. Let $wtw^{-1}=(m_1,\dots,m_l)$, so that $g_i(m_i)\ne0$ ($i=1,\dots,l$).  
Since $q_i=0$ for $i\ge2$, we have also that $m_i=0$ for all $i\ge2$. 
Hence $wtw^{-1}=(m_1,0,\dots,0)$, $m_1\ne0$, $q_1|m_1$, as claimed.
Write $y=q_1y'$. Then we have $|y'|q_1\le q_1L$, and so $|y'|\le L$. 
Since there are $2L$ such $y'$, the support is of size at most $|W_0|2L$. 
\end{proof}
\begin{lem}\label{speclem}
\begin{enumerate}
\item The support of $\tilde{k}_N$ consists of elements $t\in T$ with 
$N\ll_L \norm{t}\ll N$. 
\item $\sup_{t\in T}|\tilde{k}_N(t)|\ll_L 1$. 
\end{enumerate}
\end{lem}
\begin{proof}
For any $t\in T$ we have 
\begin{equation}\label{eq126}
\tilde{f}*\tilde{f}^*(t)=\sum_{t'\in T}\tilde{f}(t'+t)\overline{\tilde{f}(t')}. 
\end{equation}
Thus if $\tilde{f}*\tilde{f}^*(t)\ne0$ then there is some $t'\in T$ such that both 
$\tilde{f}(t')\ne0$ and $\tilde{f}(t'+t)\ne0$. 
Thus if $t=(m_1,\dots,m_l)$, then 
each of the $m_i$ is a multiple of $q_1$. Since at least one of the $m_i$ is not zero
it follows that $\norm{t}\gg q_1\gg_LN$. To prove the upper bound,
note that we have $\norm{t'}, \norm{t+t'}\ll q_1L$ due to Lemma \ref{lem141}, 
and so $\norm{t}\ll q_1L\ll\frac{N}{L}L=N$.  
This proves the first part of the lemma. To prove the second part, let $t\in T$. If $t=0$ then $\tilde{k}_N(t)=0$ and 
there is nothing to prove. Otherwise $\tilde{k}_N(t)$ is given by \eqref{eq126}.  
Using Lemma \ref{suppsize} and the triangle inequality, 
we see that $|\tilde{k}_N(t)|\ll_L|\tilde{f}(t'+t)\overline{\tilde{f}(t')}|$ for some $t'\in T$. 
The result now follows from the fact that $|\tilde{f}|\ll_L 1$, due to Lemma \ref{lem141}. 
\end{proof}

\begin{lem}\label{speclem7}
\begin{enumerate}
\item $\tilde{f}*\tilde{f}^*(0)\ll L^{2l-1}$.
\item $\hat{\omega}_s(k_N)\gg L^{2l}$.
\end{enumerate}
\end{lem}
\begin{proof}
We have
$$
\tilde{f}*\tilde{f^*}(0)=\sum_{t\in T}|\tilde{f}(t)|^2.
$$
If $t=(m_1,\dots,m_l)$ is such that $\tilde{f}(t)\ne0$ then at least one of the $m_i$'s is not zero and so   
$g(t)\le (2L+1)^{l-1}$. Similarly $g(wtw^{-1})\le(2L+1)^{l-1}$ for each $w\in W_0$
and so $\tilde{f}(t)=\frac{1}{|W_0|}\sum_{w\in W_0}g(wtw^{-1})\le (2L+1)^{l-1}$. Combining this with 
Lemma \ref{suppsize} the first assertion of the lemma follows. 
To prove the second assertion, recall that we have $|\hat\omega_s(f)|\gg L^l$,
due to \eqref{A}. Combining this with the first assertion finishes the proof.
  
\end{proof}
\subsection{Geometric properties of $k_N$}
We start by bounding $k_N$. Using Theorem \ref{thm2} we have
\begin{equation}\label{expansion}
\begin{aligned}
k_N(x)&=k_N*1_{K_p}(x)\\
&=k_N*\int_{\hat{T}}\omega_s(x)d\mu(s)\\
&=\int_{\hat{T}}\hat{\omega}_s(k_N)\omega_s(x)d\mu(s)\\
&=\int_{\hat{T}}\sum_{t\in T}s(t)\tilde{k}_N(t)\omega_s(x)d\mu(s)\\
&=\sum_{t\in T}\tilde{k}_N(t)\int_{\hat{T}}s(t)\omega_s(x)d\mu(s).
\end{aligned}
\end{equation}
\begin{lem}\label{lem8}
There exist $C>0$ and $\kappa>0$ such that for any $x\in G_p$ we have 
$$
|\int_{\hat{T}}s(t)\omega_s(x)d\mu(s)|\le Ce^{-\kappa\norm{t}}.
$$
\end{lem}

\begin{proof}
Let $s\in\hat{T}$ be a nonsingular character (see \eqref{nonsing}). Then up to a positive constant we have 
$$
\omega_s(x)=\omega_s(z_0^{-1})=\delta(t_0)^{-\frac{1}{2}}\sum_{w\in W_0}(ws,t_0)c(ws), 
$$
for every $z_0\in Z^{++}$ and $t_0=\nu(z_0)$, due to \eqref{mf}. If $s_0$ is singular, then $\omega_{s_0}(x)$ is
just the limit $\lim_{s\to s_0}\omega_s(x)$ over nonsingular $s$. Since the limit exists, we continue to write
the same formula for singular characters as well. 
By Theorem \ref{planch} we have
$$
\int_{\hat{T}}s(t)\sum_{w\in W_0}(ws,t_0)c(ws)d\mu(s)=
\int_{\hat{T}}s(t)\sum_{w\in W_0}(ws,t_0)c(ws)\frac{1}{|c(s)|^2}ds.
$$
Fix $w\in W_0$. We have $|c(s)|^2=c(ws)c(ws^{-1})$, due to \eqref{inv}. Thus 
$$
\int_{\hat{T}}s(t)(ws,t_0)c(ws)\frac{1}{|c(s)|^2}ds=
\int_{\hat{T}}s(t)(ws,t_0)\frac{1}{c(ws^{-1})}ds.
$$
The function $s\mapsto (ws,t_0)\frac{1}{c(ws^{-1})}$ is real analytic on $\hat{T}$
(see the discussion following \eqref{hc}).
Thus, viewing  
it as a function on the torus $(\R/\Z)^l$, it can be extended 
holomorphically to a $\kappa_w$-neighborhood $U$ of $(\R/\Z)^l$
in $(\C/\Z)^l$ for some positive constant $\kappa_w$, which is clearly independent of $t_0$. Thus by the
Paley-Wiener Lemma on exponential decay of Fourier coefficients we have 
$$
|\int_{\hat{T}}s(t)(ws,t_0)\frac{1}{c(ws^{-1})}ds|\ll C_w(t_0) e^{-\kappa_w\norm{t}},
$$
where $C_w(t_0)=\sup_{s\in U}({s\mapsto (ws,t_0)\frac{1}{c(ws^{-1})}})$. Taking $\kappa_w$ small 
enough, we can arrange that $\delta(t_0)^{-\frac{1}{2}}C_w(t_0)$ is uniformly bounded in terms of $t_0$, i.e.
in terms of $x$, and so  
$$
|\delta(t_0)^{-\frac{1}{2}}\int_{\hat{T}}s(t)(ws,t_0)\frac{1}{c(ws^{-1})}ds|\ll e^{-\kappa_w\norm{t}}.
$$
The result now follows since our integral is just a finite sum of such expressions.
\end{proof}

\begin{cor}\label{thm35}
There exists a positive constant $\delta=\delta(L)>0$ such that for all $x\in G_p$ and $N$ large enough, 
$$
|k_N(x)|\le e^{-\delta N}.
$$
\end{cor}

\begin{proof}
By Lemma \ref{speclem} we have that $\tilde{k}_N(t)\ne0$ only if $C'N\le\norm{t}\le CN$
for some constants $C',C>0$ (depending on $L$).
Note that the number of such $t\in T$ is bounded 
(up to a constant) by $N^l$. 
By Lemma \ref{lem8} we have that for any such $t\in T$ 
$$
|\int_{\hat{T}}s(t)\omega_s(x)d\mu(s)|\ll e^{-\kappa\norm{t}}\le e^{-\kappa C'N},
$$
for some constant $\kappa>0$. From the second part of 
Lemma \ref{speclem} we have $\tilde{k}_N(t)\ll_L 1$. 
Thus we have by \eqref{expansion} that
$$
|k_N(x)|\ll_L N^le^{-C'\kappa N}.
$$
The result follows by absorbing the polynomial coefficient $N^l$ into the exponent. 
\end{proof}

Next we estimate the support of $k_N$. 

Let $p_t:\hat{T}\to \C$ be defined by 
$$
p_t(s)=\sum_{w\in W_0}(ws,t).
$$
We have 
\begin{equation}\label{expansion2}
\begin{aligned}
k_N(x)&=\sum_{t\in T/W_0}\sum_{w\in W_0}\tilde{k}_N(wt)\int_{\hat{T}}(ws,t)\omega_s(x)d\mu(s)\\
&=\sum_{t\in T/W_0}\tilde{k}_N(t)\int_{\hat{T}}p_t(s)\omega_s(x)d\mu(s).
\end{aligned}
\end{equation}
\begin{lem}\label{lem199}
Given any $t_0\in T$ there exists $k_{t_0}\in L(G_p,K_p)$ such that for all $s\in\hat{T}$
$$
k_{t_0}*\omega_s=p_{t_0}(s)\omega_s.
$$
\end{lem}

\begin{proof}
By the Satake isomorphism it is enough to show that there is an element $\tilde{k}_{t_0}\in L(T)^{W_0}$ such that 
for all $s\in \hat{T}$ we have
$$
p_{t_0}(s)=\sum_t \tilde{k}_{t_0}(t)s(t).
$$
It is clear that $\tilde{k}_{t_0}(t)=\delta_{t\in W_0 t_0}$ (here $W_0t_0$ is the orbit of $t_0$ under the action of $W_0$)
 is such an element. 
\end{proof}

Following $\cite{M}$, we denote the element $\tilde{k}_t=\delta_{W_0 t}\in L(T)^{W_0}$ from 
Lemma \ref{lem199} by $\left<t\right>$. 
We continue to denote by $k_t$ the element of $L(G_p,K_p)$ that corresponds to $\left<t\right>$.
We now consider $T^{++}$ as an ordered set, viewing it as a subset of the ordered set 
$(T,\le)$ (see definition \ref{ordering}). Then we can write $T^{++}$ as a sequence
$$
T^{++}=\{t_i\}_{i=0}^\infty,
$$ 
with $0=t_0\le t_1\le\dots$.
Let $m$ be the least integer such that $t_1,\dots,t_m$ 
generate $T^{++}$ as a semi group (it is in fact easy to see that $m=l$). For each element $t$ in 
$T^{++}$, let $\chi_t$ 
denote the characteristic function of $K_p\nu^{-1}(t)K_p$. 

\begin{lem}\label{lem200}
For each $t\in T^{++}$ there is a
polynomial $P_t$ in $m$ variables, of total degree $\ll \norm{t}$ such 
that $k_t=P_t(\chi_{t_1},\dots,\chi_{t_m})$. 
\end{lem}

\begin{proof}
We shall use the following two facts.
\begin{enumerate}
\item For any $t_1,t_2\in T^{++}$ we have 
\begin{equation}\label{eq245}
\left<t_1\right>*\left<t_2\right>=\left<t_1t_2\right>+\sum_{t'\in T^{++}, t'<t_1t_2}c_{t'}\left<t'\right>.
\end{equation}
In fact, for each $t\in T^{++}$ let $\theta_t=\delta(t)^{\frac{1}{2}}\left<t\right>$,
and $M_t=\sum_{t'<t}\Z\theta_{t'}$. It is shown in the proof of {\cite[(3.3.14)]{M}}
that 
$$
\theta_{t_1}*\theta_{t_2}\equiv\theta_{t_1t_2} \mod M_t.
$$
In particular we obtain \eqref{eq245}.
\item For any $t\in T^{++}$ we have 
$$
\tilde{\chi}_t=\delta(t)^{\frac{1}{2}}\left<t\right>+\sum_{t'<t,t'\in T^{++}}\tilde{\chi}_t(t')\left<t'\right>
$$
({\cite[(3.3.8')]{M}}).
\end{enumerate}
Let $t=t_{i_1}\cdots t_{i_k}\in T^{++}$. It follows from \eqref{eq245} that 
$$
\left<t\right>=\left<t_{i_1}\right>*\dots*\left<t_{i_k}\right>-\sum_{t'\in T^{++},t'<t}c_{t'}\left<t'\right>,
$$
and so by induction and by Lemma \ref{lem-same3} we have that $\left<t\right>$ is a polynomial
in the $\left<t_i\right>$'s with degree $\ll k$. Applying the second fact above to $t=t_i$ we have
that each $\left<t_i\right>$ is a \emph{linear} polynomial in $\tilde{\chi}_{t_1},\dots,\tilde{\chi}_{t_m}$.
Thus $\left<t\right>$ is a polynomial of degree $\ll k$ in $\tilde{\chi}_{t_1},\dots,\tilde{\chi}_{t_m}$. Thus by Lemma 
\ref{lem-same} and Lemma \ref{lem-same3} it follows that if $t(0)=\sum_{i=1}^lm_i(t)a_i^{\vee}$
then the degree is $\ll m_1(t)$. But $t\in T^{++}\mapsto m_1(t)$ is equivalent to $t\mapsto \norm{t}$, and 
so the degree is $\ll \norm{t}$. The result now follows by the Satake isomorphism (Theorem \ref{thm2}).  
\end{proof}

\begin{lem}\label{lem11}
Let $z_0\in Z^{++}$ and $t_0=\nu(z_0)$. If $\norm{t}\ll\norm{t_0}$ then 
$$
\int_{\hat{T}}p_t(s)\omega_s(z_0^{-1})d\mu(s)=0,
$$
\end{lem}
\begin{proof}
We have 
\begin{align*}
\int_{\hat{T}}p_t(s)\omega_s(z_0^{-1})d\mu(s)&=
\int_{\hat{T}}k_t*\omega_s (z_0^{-1})d\mu(s)\\
&=\int_{G_p} k_t(z_0^{-1}g)\int_{\hat{T}}\omega_s(g^{-1})d\mu(s)dg\\
&=k_t*\chi_0(z_0^{-1})=k_t(z_0^{-1}).
\end{align*}
By Lemma \ref{lem200} we have that $k_t$ is a polynomial of degree $\ll \norm{t}$ in 
$\chi_{t_1},\dots,\chi_{t_m}$, where $t_1,\dots, t_m$ is a fixed choice of 
generators of the semigroup $T^{++}$. For simplicity write $\chi_{t_i}=\chi_i$. 
Let $\chi_{i_1}*\cdots*\chi_{i_d}$ one of the monomials of our polynomial. 
We claim that $\chi_{i_1}*\cdots*\chi_{i_d}(z_0^{-1})\ne0$ only if $z_0^{-1}\in\prod_{j=1}^dKt_{i_j}K$.
Indeed, we have 
$$
\chi_{i_1}*\cdots*\chi_{i_d}(z_0^{-1})=\int_{G_p}\chi_{i_1}*\dots*\chi_{i_{d-1}}(g^{-1})\chi_{i_d}(z_0^{-1}g)dg.
$$
For the integrand to not vanish we need $z_0^{-1}g\in Kt_{i_d}K$ and $\chi_{i_1}*\dots*\chi_{i_{d-1}}(g^{-1})\ne0$. 
Thus by induction $g^{-1}\in \prod_{j=1}^{d-1}Kt_{i_j}K$ and so $z_0^{-1}\in \prod_{j=1}^dKt_{i_j}K$ as needed.
Thus if $\chi_{i_1}*\cdots*\chi_{i_d}(z_0^{-1})\ne0$ we have that 
$$
p^{C'\norm{t_0}}\le \d(z_0^{-1})\le\prod_{j=1}^d\d(t_{i_j})\le p^{C(\sum_{j=1}^d\norm{t_{i_j}})},
$$
by Lemma \ref{lem-lior} and Lemma \ref{lem-d}, and so 
$$
\norm{t_0}\ll \sum_{j=1}^d\norm{t_{i_j}}\ll d\ll \norm{t}.
$$
Thus we have that if $\norm{t}\ll\norm{t_0}$ then $\chi_{t_1},\dots,\chi_{t_m}(z_0^{-1})=0$. 
Since this holds for each of the monomials we have that $k_t*\chi_0(z_0^{-1})=0$ whenever 
$\norm{t}\ll\norm{t_0}$, as needed.  
\end{proof}

\begin{cor}\label{thm-supp}
$k_N$ is supported on elements $x\in K_pz_0K_p$ ($z_0\in Z^{++}$) with $\norm{\nu(z_0)}\ll N$.
\end{cor}
\begin{proof}
By Lemma \ref{speclem} there exists $C>0$ such that $\tilde{k}_N(t)$=0 whenever
$\norm{t}\ge CN$. By Lemma \ref{lem11} there exists $C'>0$ such that 
$\int_{\hat{T}}p_t(s)\omega_s(x)d\mu(s)=0$ whenever $\norm{\nu(z_0)}>C'\norm{t}$. 
Thus by \eqref{expansion2}, we have that $k_N(x)=0$ whenever $\norm{\nu(z_0)}\ge CC'N$. 
\end{proof}

\begin{proof}[Proof of Proposition \ref{mainlem}]
Let $0<\epsilon<1$ and $\phi\in L^2(\X)$ an 
$L(G_p,K_p)$-eigenfunction with spectral parameter $s\in S$.
Let $k_N=k_{L,N}$ be the corresponding kernel as defined in \eqref{kerdef}. 
By Lemma \ref{speclem7} there exists $L\in\N$ such that 
$$
[\tilde{f}*\tilde{f}^*(0)]^{-1}\omega_s(k_N)>\frac{1}{\epsilon},
$$
whenever $N$ is large enough in terms of $\epsilon$. 
Let $K_N=[\tilde{f}*\tilde{f}^*(0)]^{-1}k_N$. 
Then $K_N$ satisfies the last two properties by construction (for the last property 
we have used in particular that $f*f^*$ acts as a positive operator on $L^2(\X)$). 
The first and second properties follow by Corollary \ref{thm-supp} and Corollary \ref{thm35} respectively.      
\end{proof}

\section{Proof of Theorem \ref{main-main}}
\subsection{The partition $\P$}\label{partition}
We keep denoting $G=\SL_n(\R)$, $\Gamma=\SL_n(\Z)$, and $X=\X$. 
We fix a left-invariant Riemannian metric on $G$. Recall our identification \eqref{decomp}
between $X$ and the double quotient $\Gamma_p\backslash G\times G_p/K_p$. We use the 
following notational conventions.
If $U\subset G$ is any subset we denote by $\overline{U}$ its image in $X$ under the natural projection. 
If $U\subset G\times G_p$ we denote by 
$\overline{U}$ the image of $U$ in $\Gamma_p\backslash G\times G_p/K_p$.
Notice that for any $U\subset G$, the sets $\overline{U}\subset \Gamma\backslash G$ and 
$\overline{U\times \{1\}}\subset \Gamma_p\backslash G\times G_p/K_p$ coincide under our identification 
\eqref{decomp}. Thus if $U\subset G$ then we use the notation $\overline{U}$ for either of these two sets
interchangeably. If $x\in G$ and $b\in G_p$ we denote by $xb$ the element $(x,b)\in G\times G_p$.
We say that an element $a$ of the diagonal group $A$ is \emph{regular} if it has distinct 
entries.
We consider measures on $X$ invariant under the action of such $a$. If $\P$ is a partition of 
$X$, let $\P_N$ be the (symmetric) \emph{$N$-th refinement of $\P$ under the action of $a$}:
$$
\P_N:=\bigvee_{i=-N}^Na^i\P. 
$$ 
As before we denote $G_p=\SL_n(\Q_p), K_p=\SL_n(\Z_p)$ and we consider the action of the 
Hecke algebra $L(G_p,K_p)$ on $L^2(X)$ through the double quotient decomposition \eqref{decomp}.  

In this section we prove the following lemma. 
  
\begin{lem}\label{thm-gp}
Let $a\in A$ be a regular element, and $\mu$ an $a$-invariant probability measure on $X$. 
For any compact identity neighborhood $\Omega\subset G$ and $r>0$ there 
exists a sequence 
of identity neighborhoods 
$B_N\subset G$, and a partition $\P=\P(\Omega,r)$ of $X$, satisfying:
\begin{enumerate}
\item For any $x,y\in\Omega$, the number of cosets $bK_p$ ($b\in G_p$) with denominator 
$\le p^{rN}$ such that 
$\overline{xB_Nb}\cap \overline{yB_N}\ne\emptyset$ is at most $N^{O_r(1)}$.

\item There exists $c\in\N$ depending only on $r$ such that 
the intersection of $\overline{\Omega}$ with any element of
$\P_{cN}$ is contained, up to a $\mu$-null set, in a translate 
$\overline{xB_N}$, $x\in\Omega$. 
\end{enumerate}
Here the implied constants may depend on $r$ and $\Omega$, but not on $N$. 
\end{lem}
We remark that the constant 
$r$ in Lemma \ref{thm-gp} will be taken to be the 
constant $r$ from the first item of Proposition \ref{mainlem}. Thus $r$, and 
accordingly the constant $c$ in the second item of the lemma, should be thought of as an 
\emph{absolute} constant.   

Following \cite{SV}, our dynamical balls $B_N$ will be thickened compact pieces of $A$.
Given a compact neighborhood of the identity $C\subset A$  
let $B(C,\epsilon)$ be an $\epsilon$-neighborhood of $C$ inside $G$. The proof 
of Lemma \ref{thm-gp} is based on the following results. 
 
\begin{lem}[{\bf{Silberman--Venkatesh}~\cite{SV}, Lemma 4.4}]\label{SV4.4}
Let $\Omega\subset G$ be a compact neighborhood of the identity.
For $c>0$ sufficiently large, depending only on $n$, and $c'>0$ sufficiently small, depending on $\Omega$, 
for any $g\in\Omega$ the set of $\gamma\in\Gamma_p$ such that
$$
\inf\{\text{dist}(\gamma,t)\mid t\in g(A\cap\Omega\Omega^{-1}) g^{-1}\}\le\epsilon, \d(\gamma)\le M 
$$
is contained in a $\Q$-torus $T<\SL_n(\Q)$, provided that $\epsilon M^c\le c'$
\end{lem}
\begin{lem}[\cite{Lin}, (7.51)]\label{EL4.9}
Let $a$ be an element of $A$ and $\mu$ an $a$-invariant probability measure on $X$. 
Let $F\subset X$ a compact subset, and $\delta>0$. 
There exists a countable partition $\P$ of $X$ with finite entropy, containing $X\setminus F$
as one of its elements, satisfying the following property. For every element $E\subset F$ 
in the $N$-th refinement
$\P_N$ of $\P$, there exists $x\in F$ so that up to a $\mu$-null set,  
$$
E\subset x\bigcap_{k=-N}^N a^{-k}B_\delta^Ga^k.
$$
Here $B_\delta^G\subset G$ is the open ball of radius $\delta$ around the identity in $G$. 
\end{lem}

\begin{lem}\label{lem-z}
Let $N\in\N$ and $\delta>0$ which is small enough in terms of $a$. 
Then there exists some $\alpha>0$, depending only on $a$ such that 
$$
\bigcap_{k=-N}^Na^{-k}B_{\delta}^Ga^{k}\subset B(C,\kappa e^{-\alpha N}),
$$
for some $\kappa\ll\delta$, and $C\subset A$ some compact subset with diameter $\ll\delta$. 
\end{lem}
For the proof of Lemma \ref{lem-z} we use the following notation. Define $G^-$ to be the 
set of all elements $g\in G$ such that $a^iga^{-i}\to e$ as $i\to\infty$ and $G^+$ the
set of all elements $g\in G$ such that $a^{-i}ga^i\to e$. Then $G^+$ and $G^-$ are closed 
subgroups of $G$, normalized by $a, a^{-1}$.
If $\g^+$ and $\g^-$ are the Lie algebras corresponding to $G^+$ and $G^-$ and $\a$ is the Lie 
algebra corresponding to the centralizer of $a$ (which is in our case the Lie algebra of $A$)
then 
$$
\g=\a\oplus\g^+\oplus \g^-.
$$

\begin{proof}[Proof of Lemma \ref{lem-z}]
Let $\theta:G\to G$ be conjugation by $a$. 
Let $U\subset\g$ be a neighborhood of $0\in\g$ such that $\exp|U$ is a diffeomorphism onto 
its image. 
Let $\delta>0$ and set $U'=(\exp|_U)^{-1}(B_\delta^G)$. We choose $\delta>0$ small
enough so that $\theta^{\pm1}U'\subset U$.
Let $g\in\cap_{-N}^N\theta^iB_\delta^G$. We claim that $g=\exp(\theta^NZ_N)$ for some
$Z_N\in U'$ such that $\theta^NZ_N\in U'$. Indeed, by induction assume that we have 
$$
g=\exp(Z_0)=\exp(\theta Z_1)=\dots=\exp(\theta^{N-1}Z_{N-1}),
$$
for $Z_i\in U'$ with $\theta^i Z_i\in U'$, as above. Since $g\in\bigcap_{i={-N}}^N\theta^iB_\delta^G$,
there exists $Z_N\in U'$ so that $g=\exp(\theta^NZ_N)$. 
We have $g=\exp(\theta^NZ_N)=\exp(\theta^{N-1}Z_{N-1})$ and so $\exp(\theta Z_N)=\exp(Z_{N-1})\in B_\delta^G$. 
On the other hand we have that $\theta Z_N\in U$, since $Z_N\in U'$, and so $\theta Z_N\in U'$. Similarly, 
$\exp(\theta^N Z_N)=\exp(\theta^{N-2}Z_{N-2})$, and so $\exp(\theta^2Z_N)\in B_\delta^G$. But 
$\theta Z_N\in U'$, so that $\theta^2 Z_N\in U$, and so $\theta^2 Z_N\in U'$. By applying this argument 
inductively it follows that $\theta^NZ_N\in U'$.

We can write $g$ in the following three forms:
\begin{align*}
g&=\exp(H_0+X_0+Y_0)\\
&=\exp(\theta^N(H_N+X_N+Y_N))\\
&=\exp(\theta^{-N}(H_{-N}+X_{-N}+Y_{-N})), 
\end{align*}
where $H_i\in\a, X_i\in\g^-,Y_i\in\g^+$, and $H_i+X_i+Y_i\in U',\theta^i(H_i+X_i+Y_i)\in U'$ ($i=\pm N,0$).
Since both $a$ and $a^{-1}$ normalize both $G^+$ and $G^-$ and by the injectivity of $\exp|_U$ we have 
that $X_0= \theta^NX_N$ and that $Y_0=\theta^{-N}Y_{-N}$. Thus there are constants $c,\alpha>0$ depending
only on $a$ such that $||X_0||,||Y_0||\le c\delta e^{-\alpha N}$, 
and $||H_0||\le c\delta$ (see e.g. \cite{Lin}, Lemma 7.29). Thus $H_0+X_0+Y_0$ 
belongs to a $c\delta e^{-\alpha N}$-neighborhood of a $c'\delta$-neighborhood of $\a$ (where $c'$ is some
positive constant depending only 
on $a$). The claim for neighborhoods in the group
$G$ follows from the fact that $\exp$ is $c''$-Lipschitz, for some
$c''>0$, on $U$.   
\end{proof}

\begin{proof}[Proof of Lemma \ref{thm-gp}.]
Let $\delta>0$. By Lemma \ref{EL4.9} there exists a partition $\P$ containing $X\setminus\overline{\Omega}$ so that
for every $E\in\P_N$, $E\subset \overline{\Omega}$, there exists $x\in\Omega$ such that, up to a null set, 
$$
E\subset\overline{x\bigcap_{k=-N}^N\theta^kB_\delta^G}.
$$
Here $\theta$ stands for conjugation by $a$.
Taking $\delta$ small enough to satisfy Lemma \ref{lem-z} it follows that there exists
$\alpha>0$ depending only on $a$ so that 
$$
\bigcap_{k=-N}^N\theta^kB_\delta^G\subset B(C,\kappa e^{-\alpha N}),
$$
where $\kappa\ll\delta$ and $C\subset A$ with diameter $\ll\delta$. 
Thus we have a partition $\P=\P(\delta)$, 
and $\alpha>0$ depending only on $a$ such that for every $c\in \N$, for every 
$E\subset \overline{\Omega}$ in the $cN$-th refinement of $\P$, 
$$
E\subset \overline{x B(C,\kappa e^{-\alpha cN})}
$$
for some $x\in\Omega$, $C\subset A$ of diameter $\ll \delta$,
and $\kappa\ll\delta$. Thus it suffices to show that $\delta$ and $c$ 
can be chosen in a way that if we define $\epsilon_N:=\kappa e^{-\alpha cN}$, then for every 
$x,y\in \Omega$,
\begin{equation}\label{eq100}
\overline{xB(C,\epsilon_N)b}\cap\overline{yB(C,\epsilon_N)}\ne\emptyset
\end{equation}
for at most $N^{O_r(1)}$ cosets $bK_p$ ($b\in G_p$) with $\d(b)\le p^{rN}$.
Suppose that the intersection above is non-empty for such cosets $b_1K_p,\dots,b_kK_p$. 
Denote $B=B(C,\epsilon_N)$. 
Then from \eqref{eq100} we have that for each $i$ 
there exists $\gamma_i\in\Gamma_p$ such that
\begin{equation}\label{eq101}
\gamma_ixBb_i\cap yB\ne\emptyset,
\end{equation}
in $G\times G_p/K_p$. By \eqref{eq101} we have that $\gamma_i b_i\in K_p$.
Thus by replacing $b_i$ by an appropriate representative for $b_i K_p$, we 
may assume that $\gamma_i b_i=1$.
So we have $\gamma_i\in yBB^{-1}x^{-1}$ and $\d(\gamma_i)\le p^{O(rN)}$. 
Let $s_i=\gamma_i\gamma_1^{-1}$ ($i=2,\dots,n$), 
then $s_i\in yBB^{-1}BB^{-1}y^{-1}\cap \Gamma_p$. 
Taking $c$ sufficiently large and $\delta$ sufficiently small it follows 
from Lemma \ref{SV4.4} that the $s_i$ are all lying on $T\cap\SL_n(\Z[\frac{1}{p}])$, 
for some torus $T<SL_n(\Q)$. 
On the other hand 
the $s_i$ still have denominator $\le p^{O(rN)}$. Thus taking $\delta$ sufficiently small (not depending on $T$) 
the result follows by Lemma \ref{counting} below. 
\end{proof}

\begin{lem}\label{counting}
Let $T<\SL_n(\Q)$ be a torus (that is, a commutative subgroup all elements of which are diagonalizable over $\C$). Let 
$U\subset G$ be an identity neighborhood. Assume that for every $g\in U^{-1}U$ and $\lambda\in \C$
an eigenvalue of $g$, we have $|\lambda-1|<\frac{1}{2}$. 
Then there are at most $N^{O(1)}$ elements 
$s\in U\cap T\cap\SL_n(\Z[\frac{1}{p}])$ with denominator $\d(s)\le p^N$. 
\end{lem}

\begin{proof}
Let $L$ be the splitting field of $T$ over $\Q$, and $\O_L$ its ring of integers.
Denote $T(\Z[\frac{1}{p}]):=T\cap\SL_n(\Z[\frac{1}{p}])$.
Then there exists a matrix $\sigma\in \GL_n(L)$ such that for every $s\in T(\Z[\frac{1}{p}])$ we have
\begin{equation}\label{eq200}
\sigma s\sigma^{-1}=\text{diag}(\lambda_1,\dots,\lambda_n),
\end{equation}
where $\lambda_i=\lambda_i(s)\in L^\times$ are the eigenvalues of $s$.
This gives an embedding of $T(\Z[\frac{1}{p}])$ into $(L^\times)^n$, sending $s$
to $(\lambda_1,\dots,\lambda_n)$. Consider 
the map which takes $s\in U\cap T(\Z[\frac{1}{p}])$ to the 
tuple $(\lambda_1\O_L,\dots,\lambda_n\O_L)$ of fractional ideals of $\O_L$. 
We claim that this map is injective. 
Indeed, let $s,s'\in U\cap T(\Z[\frac{1}{p}])$, and let 
$\theta_i=\lambda_i^{-1}\lambda_i'$.
If 
$$
\lambda_1\O_L,\dots,\lambda_n\O_L=\lambda'_1\O_L,\dots,\lambda'_1\O_L,
$$ 
then $\theta_1,\dots,\theta_n\in\O_L^\times$. 
Note that by \eqref{eq200}, we have that 
$\theta_1,\dots,\theta_n$ are the eigenvalues of $s^{-1}s'\in U^{-1}U\cap T(\Z[\frac{1}{p}])$.
Since $s^{-1}s'\in U^{-1}U$, our assumption on $U$ implies that for each $i$,
\begin{equation}\label{eq201}
|\theta_i-1|<\frac{1}{2}.
\end{equation}
Note also that since $s^{-1}s'\in T(\Z[\frac{1}{p}])$, 
we have in particular that the characteristic polynomial of $s^{-1}s'$ is a monic polynomial 
with rational coefficients. Since $\theta_1,\dots,\theta_n$ are the roots of this polynomial, 
we have that for each $i$, all of the conjugates of $\theta_i$ are 
contained in $\{\theta_1,\dots,\theta_n\}$. Thus, since any element of $\O_L\setminus\{1\}$ 
has at least one conjugate at distance $\ge\frac{1}{2}$ from $1$ (for example by looking at the discriminant), 
it follows from \eqref{eq201} that $\theta_i=1$. 
This proves the claim. Thus it remains to show that there are at most $N^{O(1)}$ tuples of fractional
ideals corresponding to elements $s\in U\cap T(\Z[\frac{1}{p}])$ with denominator at most $p^N$.
For this notice first that if we write $\lambda_i\O_L$ as a reduced quotient of 
integral ideals of $\O_L$, then since $\lambda_i$ belongs to the integral closure of $\Z[\frac{1}{p}]$
in $\O_L$, any prime ideal that appears in the denominator must divide $p$. Also since the denominator
of $s$ is at most $p^N$, any such prime ideal appears with multiplicity at most $O(N)$.
Since $\prod_{i=1}^n\lambda_i=1$, it follows that the ideals appearing in the nominator must satisfy these
properties as well. Since there are at most $O(1)$ prime ideals dividing $p$ (note that $[L:\Q]$ is at most $O_n(1)$),  
we have at most $N^{O(1)}$ choices for our fractional ideal at each coordinate, hence at most 
$N^{O(1)}$ tuples of such ideals.  
\end{proof}

\subsection{Proof of Theorem \ref{main-main}}

We keep the notation of the previous section. Instead of working with the definition of entropy we use the following
proposition which gives a criterion for positive entropy on almost every ergodic component. The reader who is not familiar with entropy can take this as a definition. See Appendix \ref{sec.app1} (Proposition \ref{posent}) for the necessary background and a proof.  
\begin{prop}\label{entdef}
Let $a:X\to X$ be a measurable map. Let $\mu$ be an $a$-invariant 
probability measure on $X$. Then $a$ acts with 
positive entropy on almost every ergodic component
if for any $\eta>0$ there exists a partition $\P$ of $X$ satisfying 
the following condition for some constants $\delta>0,c\in\N$:
For any $N$ large enough, 
if $J\subset X$ is a union of $d$ elements of the $cN$-th refinement 
$\P_{cN}$ of $\P$ (under $a$), 
of total measure $\mu(J)>\eta$, then $d>e^{\delta N}$.
\end{prop}

Our proof of Theorem \ref{main-main} is based 
on analyzing an expression of the form 
\begin{equation}\label{eq203}
\left<\phi1_{J}*K_N,\phi1_{J}\right>,
\end{equation}
for a subset $J\subset X$ as in Proposition \ref{entdef}, $\phi\in L^2(\X)$ a suitable 
$L(G_p,K_p)$-eigenfunction, and $K_N$ the kernel constructed in section \ref{sec-ker}.
\begin{lem}\label{lem102}
Let $E,E'\subset G$ be two measurable subsets of $G$, contained in some fixed fundamental domain $F$
for $\X$.
Let $k\in L(G_p,K_p)$, and let $v=v(k)$ be the number of cosets $bK_p$ ($b\in G_p$) such that 
\begin{enumerate}
\item $K_p b^{-1} K_p$ is in the support of $k$,
\item $\overline{E'b}\cap\overline{E}\ne\emptyset$ in $X$. 
\end{enumerate}
Then for any $f\in L^2(X)$ we have
\begin{equation}\label{eqvee}
|\left<f1_{\bar{E}}*k,f1_{\bar{E'}}\right>|\le v\norm{k}_{\infty}\norm{f}_{L^2(\bar{E})}\norm{f}_{L^2(\bar{E'})}.
\end{equation}
\end{lem}

\begin{proof}
For any function $h\in L^2(X)$ we have 
$$
h*k(x)=\int_{G_p}k(y^{-1})h(xy)dy,
$$
where for any $x=\overline{(x_\infty,1)}\in X$ with $x_\infty\in F$, 
$y\in G_p$, $xy=\overline{(x_\infty,y)}$.
The union of the supports of $k$ and $g\mapsto k(g^{-1})$ is a finite union of 
double cosets $\xi_1,\dots, \xi_M$, $\xi_i=K_pz_i K_p$, and so we can 
decompose the integral above as follows.
$$
h*k(x)=\sum_{i=1}^M\int_{\xi_i}k(y^{-1})h(xy)dy=\sum_{i=1}^M k(z_i^{-1})\int_{\xi_i}h(xy)dy.
$$
Each of the double cosets $\xi_i$ is a finite union of disjoint right cosets
$$
\xi_i=\bigcup_{b\in B_{\xi_i}}bK_p,
$$  
and so  
$$
h*k(x)=\sum_{i=1}^M k(z_i^{-1})\sum_{b\in B_{\xi_i}}h(xb).
$$
Applying this equality to $h=f1_{\bar{E}}$ it follows that 
\begin{align*}
|\left<f1_{\bar{E}}*k,f1_{\bar{E'}}\right>|\le
&\sum_{i=1}^M |k(z_i^{-1})|\sum_{b\in B_{\xi_i}}\int_X\abs{f(xb)}\abs{f(x)}1_{\bar{E}}(xb)1_{\bar{E'}}(x)dx.
\end{align*}
Since the last integral vanishes unless $\bar{E}\cap\bar{E'b}\ne\emptyset$, applying Cauchy-Schwartz 
yields \eqref{eqvee}. 
\end{proof}

\begin{proof}[Proof of Theorem \ref{main-main}]
We will prove that the criterion in Proposition \ref{entdef} is satisfied.
To prove this, let $0<\eta<1$. We first use Lemma \ref{thm-gp} to obtain 
a partition of $X$.
Let $\Omega\subset G$ be a compact identity neighborhood 
of measure $\mu(\overline{\Omega})\ge1-\frac{\eta}{2}$. Let $r$ be a constant
as in the first item of Proposition \ref{mainlem}. 
Recall that $r$ depends only on $G$, $G_p$ and 
$K_p$, so we view it as an absolute constant. 
Applying Lemma \ref{thm-gp} 
with these $\Omega$ and $r$ we obtain a partition $\P=\P(\Omega,r)$ of $X$,
and a sequence of identity neighborhoods $B_N\subset G$
satisfying the two properties from that lemma. 
Let $J\subset X$ be a union 
$$
J=\bigcup_{j=1}^dE_j
$$ 
of elements $E_1,\dots,E_d\in\P_{cN}$ 
of the $cN$-refinement
of $\P$, with total mass $\mu(J)>\eta$. Here $c$ is the constant from the second item 
of Lemma \ref{thm-gp}. Recall that $c$ depends only on $r$ so we view $c$ as an absolute
constant as well. 
Put $J'=J\cap \overline{\Omega}$ and $E'_j=E_j\cap \overline{\Omega}$.
Since $\mu(\overline{\Omega})>1-\frac{\eta}{2}$, we have $\mu(J')>\frac{\eta}{2}$, and 
so there exists $i_0$ such that $\lambda\mu_{i_0}(J')>\frac{\eta}{2}$ as well. Let
$K_N\in L(G_p,K_p)$ be the kernel from Proposition \ref{mainlem} 
corresponding to $\epsilon:=\frac{\eta}{2\lambda}$ and $\phi_{i_0}$. 
Consider the inner product 
\begin{equation}\label{expression}
\left<\phi_{i_0}1_{J'}*K_N,\phi_{i_0}1_{J'}\right>=\sum_{j,k}\left<\phi_{i_0}1_{E_j'}*K_N,\phi_{i_0}1_{E_k'}\right>.
\end{equation}

By the first property of Lemma \ref{thm-gp} we have that up to a null set, each $E'_j$ 
is contained in a translate $\overline{xB_N}$ ($x\in\Omega$).
Thus for any $1\le j,k\le d$, there exist $x,y\in \Omega$ such that
$$
|\left<1_{E'_j}\phi_{i_0}*K_N,1_{E'_k}\phi_{i_0}\right>|\le 
\left<(1_{E'_j\bigcap\overline{xB_N}}|\phi_{i_0}|)*|K_N|,1_{E'_k\bigcap\overline{yB_N}}|\phi_{i_0}|\right>.
$$
By proposition \ref{mainlem} we have that $\norm{K_N}_\infty\le e^{-\delta N}$, and that  
$b\mapsto K_N(b^{-1})$  
is supported on elements with denominator 
at most $p^{rN}$. Thus it follows from Lemma \ref{lem102} that the last expression is bounded (up
to a uniform constant) by 
$$
ve^{-\delta N}\norm{\phi_{i_0}}_{L^2(E'_k)}\norm{\phi_{i_0}}_{L^2(E'_j)},
$$
where $v$ is the number of cosets $bK_p$ ($b\in G_p$) with denominator at most $p^{rN}$, such that 
$\overline{xB_Nb}\cap\overline{yB_N}\ne\emptyset$. 
But by the second property of Lemma \ref{thm-gp} there are at most $N^{O(1)}$ such cosets, and so
$$
|\left<(1_{E'_j}\phi_{i_0})*K_N,1_{E'_k}\phi_{i_0}\right>|\ll 
N^{O(1)}e^{-\delta N}\norm{\phi_{i_0}}_{L^2(E'_k)}\norm{\phi_{i_0}}_{L^2(E'_j)}.
$$     
 Applying Cauchy-Schwartz this yields
\begin{equation}\label{eq12}
\begin{aligned}
|\left<\phi_{i_0}1_{J'}*K_N,\phi_{i_0}1_{J'}\right>|&\ll
e^{-N\delta}N^{O(1)}(\sum_{j=1}^d\norm{\phi_j}_{L^2(E_j)})^2\\
&\le e^{-N\delta}N^{O(1)}d\sum_{j=1}^d||\phi_j||_{L^2(E_j)}^2\\
&=e^{-N\delta}N^{O(1)}d.
\end{aligned}
\end{equation}
To bound the left-hand side of \eqref{expression} below, we decompose 
$$
\phi_{i_0}1_{J'}=\left<\phi_{i_0}1_{J'},\phi_{i_0}\right>\phi_{i_0}+R.
$$
We have  
$$
\left<\phi_{i_0}1_{J'},\phi_{i_0}\right>=\norm{\phi_{i_0}1_{J'}}_2^2\ge\frac{\eta}{2},
$$
and so 
$$
\norm{R}_2^2=\norm{\phi_{i_0}1_{J'}}_2^2-|\left<\phi_{i_0}1_{J'},\phi_{i_0}\right>|^2
\le\norm{\phi_{i_0}1_{J'}}^2(1-\frac{\eta}{2}). 
$$
Thus by the last two properties of $K_N$ in Proposition \ref{mainlem} we obtain 
\begin{align*}
\left<\phi_{i_0}1_{J'}*K_N,\phi_{i_0}1_{J'}\right>&=
|\left<\phi_{i_0}1_{J'},\phi_{i_0}\right>|^2\left<\phi_{i_0}*K_N,\phi_{i_0}\right>+\left<R*K_N,R\right>\\
&\ge\norm{\phi_{i_0}1_{J'}}_2^4\frac{2}{\eta}-\norm{R}_2^2\\
&\ge \frac{\eta}{2}(1-(1-\frac{\eta}{2}))=(\frac{\eta}{2})^2.
\end{align*}
Combining this with \eqref{eq12} yields 
$$
d\gg\frac{\epsilon^2}{N^{O(1)}}e^{\delta N}.
$$
By absorbing the coefficient $\frac{\epsilon^2}{N^{O(1)}}$ into the exponent we see that the condition of 
Proposition \ref{entdef} is satisfied, finishing the proof of the theorem.  
\end{proof}

\section{Adjustments for the case of division algebras}\label{csa}
Let $D$ be a division algebra over $\Q$ of degree $n$, which splits over $\R$. 
Let $\O$ be a 
maximal order in $D$, and $\Gamma$ the group of all norm $1$ elements of $\O$. Since $D$ is 
$\R$-split we can view $\Gamma$ as a subgroup of $G:=\SL_n(\R)$. It is well-known that $\Gamma$
is a co-compact lattice in $G$. Assume that $D$ splits over $\Q_p$ as well (this is true for all but finitely 
many primes). Let $u_i$ ($i=1,\dots,n^2$) be a $\Z$-basis of $\O$. Using this basis
we can embed $D$ in $\M_{n^2}(\Q)$ by mapping any element $x\in D$ to the
multiplication-by-$x$ map, 
$$
m_xy=xy.
$$   
More precisely, we send $x$ to the representative matrix of $m_x$ with respect to the basis $u_i$. 
If $R$ is any subring of $\Q$, we let $D_R:=m^{-1}(M_{n^2}(R))$, and $D_R^1$ the norm $1$
elements of $D_R$. In particular, $D_\Z=\O$, and $D_\Z^1=\Gamma$. Also for any place 
$v$ we put $D_v:=D\otimes \Q_v$, and $D_v^1$ 
the norm $1$ elements of $D_v$. 
If $v=p$, then the completion $\O_p:=\Sigma_{i=1}^{n^2}\Z_pu_i$ is a maximal order 
in $D_p$ ({\cite[Corollary~11.6]{Rein}}). Let $\O_p^1$ be the group of all norm $1$ elements in $\O_p$. 
For $v=\infty, p$, fix isomorphisms 
\begin{equation}\label{eq400}
\varphi_v:D_v^1\to G_v:=\SL_n(\Q_v).
\end{equation}
Notice that $\varphi_p$ can be chosen so that 
the image of $\O_p^1$ is $K_p:=\SL_n(\Z_p)$.
Indeed, since $D$ splits over $\Q_p$ we have an isomorphism $\Phi_p:D_p\to M_n(\Q_p)$. Since 
$\O_p$ is a maximal order in $D_p$
its image under $\Phi_p$ is a maximal order in $M_n(\Q_p)$. 
As such, this image is conjugate to $M_n(\Z_p)$ ({\cite[Theorem~18.7]{Rein}}). 
Thus we can assume that $\Phi_p(\O_p)=M_n(\Z_p)$. Restricting $\Phi_p$ to $D_p^1$
yields an isomorphism $\varphi_p$ with the required property.   

Let now $\Gamma_p=D_{\Z[\frac{1}{p}]}^1$. Then similarly to \eqref{decomp}, we have the following double 
quotient decomposition. 
\begin{equation}\label{eq401}
\X\cong \Gamma_p\backslash G_\infty\times G_p/K_p.  
\end{equation}
Here $\Gamma_p$ is identified with its diagonal copy in $G_\infty\times G_p$ using \eqref{eq400}. 
As before, through this isomorphism we view $L(G_p,K_p)$ as an algebra of 
operators acting on $L^2(\X)$ by convolution on the right coordinate. In this setting, we have the following theorem.
\begin{thm}\label{main-main-csa}
Under the assumptions of Theorem \ref{main-main2}, any non-trivial element $a\in A$ acts on 
$\mu$ with positive entropy on almost every ergodic component. 
\end{thm}
The proof of Theorem \ref{main-main-csa} proceeds along the lines of the proof of Theorem \ref{main-main}. 
The kernel $K_N$ is constructed exactly as in Section \ref{sec-ker}, and Lemma \ref{thm-gp} is replaced by the
following similar lemma. 
\begin{lem}\label{thm-gp2}
Let $a\in A$ be any non-trivial element, and $\mu$ an $a$-invariant probability measure on $X$. 
For any $r>0$ there 
exist a partition $\P$ of $X$, and a sequence 
of identity neighborhoods 
$B_N\subset G$ satisfying:
\begin{enumerate}
\item There exists $c\in\N$ depending only on $r$ such that 
any element of
$\P_{cN}$ is contained in a translate 
$\overline{xB_N}$, $x\in G$. 

\item For any $x,y\in G$, the number of cosets $bK_p\in G_p/K_p$ with denominator 
$\le p^{rN}$ such that 
$\overline{xB_Nb}\cap \overline{yB_N}\ne\emptyset$ is at most $N^{O_r(1)}$.
\end{enumerate}
Here the implied constants may depend on $r$, but not on $N$. 
\end{lem} 
Fix a euclidean norm $\norm{\cdot}$ on $D\otimes_\Q\R$, and use it to define a norm 
on the Lie algebra of $G$ and thus a left-invariant Riemannian metric on $G$.    
We have the following three lemmas, similar to 
Lemma \ref{SV4.4}, Lemma \ref{EL4.9}, and Lemma \ref{lem-z}. 
\begin{lem}[{\bf{Silberman-Venkatesh}~\cite{SV}, Lemma 4.9}]\label{SV4.9}
Let $S\subset D_\R$ be a proper subalgebra. For $c>0$ sufficiently large (in fact depending only on $d$)
and for $c'>0$ sufficiently small (in fact depending only on $D$, $D_\Z$,$\norm{\cdot}$), 
the set of $x\in D$ satisfying 
\begin{equation}
\norm{x}\le R, \inf_{s\in S}{\norm{x-s}}\le\epsilon, \tilde{\d}(x)\le M, 
\end{equation}
is contained in a proper subalgbera $F\subset D$ as long as 
\begin{equation}
\epsilon R^cM^c< c'.
\end{equation} 
\end{lem}
Here $\tilde{\d}(x)=\inf\{m\in\N\mid mx\in D_\Z\}$. 
\begin{lem}\label{second}
Let $a$ be an element of $A$ and $\mu$ an $a$-invariant probability measure on $X$. 
Let $\delta>0$. 
There exists a countable partition $\P$ of $X$ with finite entropy, such that for every element $E$ 
in the $N$-th refinement
$\P_N$ of $\P$, there exists $x\in X$ so that up to a $\mu$-null set,  
$$
E\subset x\bigcap_{k=-N}^N a^{-k}B_\delta^Ga^k.
$$
Here $B_\delta^G\subset G$ is the open ball of radius $\delta$ around the identity in $G$. 
\end{lem}
Let $a$ be a non-trivial element of $A$. We extend the definition 
of the dynamical balls $B(C,\epsilon)$ from Section \ref{partition} to 
non-regular elements as well. Let $H:=C_G(a)$ be the 
centralizer of $a$ in $G$ (if $a$ is regular then $C_G(a)=A$).
For any relatively compact neighborhood of the identity in $H$, let $B(C,\epsilon)$ an 
$\epsilon$-neighborhood of $C$ in $G$. 
Define $G^+$, $G^-$ as in Section \ref{partition}, and let $\mathfrak{h}$ be the Lie algebra of $H$. 
Then we have $\g=\mathfrak{h}\oplus\g^+\oplus \g^-$, and the following version of Lemma \ref{lem-z}.  
\begin{lem}\label{third}
Let $N\in\N$ and $\delta>0$ which is small enough in terms of $a$. 
Then there exists some $\alpha>0$, depending only on $a$ such that 
$$
\bigcap_{k=-N}^Na^{-k}B_{\delta}^Ga^{k}\subset B(C,\kappa e^{-\alpha N}),
$$
for some $\kappa\ll\delta$, and $C\subset H$ some compact subset with diameter $\ll\delta$. 
\end{lem}
Using these three results, and noticing that since $D$ has prime degree any proper subalgebra of $D$ is 
a field, Lemma \ref{thm-gp2} can be proved along the lines of the proof of Lemma
\ref{thm-gp}.

\appendix\section{Entropy}\label{sec.app1}
In this appendix we review what we need from ergodic theory and entropy theory, and we prove a criterion for positive entropy on almost every ergodic component (Proposition \ref{posent}). The references are drawn from the book in progress \cite{EEW}.   

In what follows we work with locally compact second countable metric spaces. We refer to such 
a space, equipped with its Borel $\sigma$-algebra as \emph{standard Borel space}. 
We recall the following theorem regarding existence and uniqueness of conditional measures. 
\begin{thm}[\bf{Conditional measure, {\cite[Theorem~2.2]{EEW}}}]
Let $(X,\B,\mu)$ be a probability space with $(X,\B)$ standard Borel space. 
Let $\A\subset \B$ be a sub-$\sigma$-algebra. Then there
exists a subset $X'\subset X$ of full measure, belonging to $\A$, and Borel probability measures $\mu_x^\A$
for $x\in X'$ such that for every $f\in L^1(X,\B,\mu)$ we have $E(f\mid\A)(x)=\int f(y)d\mu_x^\A(y)$ for almost 
every $x$. In particular the right hand side is $\A$-measurable as a function of $x$. 
Moreover, the family of conditional measures $\mu_x^\A$ is almost everywhere uniquely determined by this 
relationship to the conditional expectation. The map $x\in X'\to \mu_x^\A$ is $\A$-measurable on $X'$. 
\end{thm}
Let $X$ be a standard Borel space and $\mu$ a probability measure on $X$. 
If $T:X\to X$ is a measure preserving map and $\E$ is the $\sigma$-algebra
of $T$ invariant sets, then (almost) every $\mu_x^\E$ is ergodic and we have the \emph{ergodic decomposition}
$$
\mu=\int\mu_x^\E d\mu(x).
$$

Let $(X,\B,\mu, T)$ an invertible measure preserving system with $X$ being a standard Borel space and $\mu$ a probability
measure. Let $\P$ be a finite or countable partition of $X$. 
The \emph{static entropy} of $\P$ is defined to be 
$$
H_\mu(\P)=-\sum_{P\in\P}\mu(P)\log\mu(P),
$$
which in the case where $\P$ is countable may be finite or infinite.
Define $\P_N=\bigvee_{i=-N}^N T^{-i}\P$ the \emph{$N$-th refinement of $\P$}.  
The ergodic theoretic entropy $h_\mu(T)$ is defined using the entropy function $H_\mu$
as follows:
\begin{defin}
Let $\P$ be a partition of $X$ with $H_\mu(\P)<\infty$. Define 
$$
h_\mu(T,\P)=\lim_{N\to\infty}\frac{1}{2N+1}H_\mu(\P_N).
$$
The \emph{ergodic theoretic entropy} is defined to be 
$$
h_\mu(T)=\sup_{H_{\mu}(\P)<\infty}h_{\mu}(T,\P).
$$
\end{defin}

If $\P$ is a partition of $X$ denote by $[x]_N$ the atom containing $x$ of $\P_N$. 
\begin{thm}[\bf{Relative Shannon--McMillan--Brieman, {\cite[Theorem~3.2]{EEW}}}]\label{smb}
Let $\P$ be a countable partition with $H_\mu(\P)<\infty$. Then for almost every $x\in X$ we have 
$$
\lim_{N\to\infty}-\frac{1}{N}\log\mu([x]_N)=h_{\mu_x^\E}(T,\P).
$$
\end{thm}
As a corollary we have the following proposition which we use in order to show positive entropy 
on almost every ergodic component.  
\begin{prop}\label{posent}
Let $(X,\B,\mu, T)$ be a measure preserving system with $X$ being a standard Borel 
probability space. Suppose that for any $\eta>0$ there exists a countable partition $\P$ with finite entropy and 
$\delta>0$ such that, for any $N$ sufficiently large, if $J$ is a collection of 
elements of the $N$-th refinement $\P_N$ of $\P$ with total measure $>\eta$, then 
$J$ has cardinality $\ge e^{\delta N}$. 
Then for almost every $x\in X$, $h_{\mu_x^\E}(T)>0$.
\end{prop}

\begin{proof}
Let $B$ be the set of $x\in X$ such that $h_{\mu_x^\E}(T)=0$ and assume 
$\mu(B)>\eta>0$. Then by assumption 
there exist a countable partition $\P$ with finite entropy, $\delta>0$, and a positive integer $N_0$, 
such that for any $N>N_0$ if $J$ is a collection of
elements of the $N$-th refinement $\P_N$ of $\P$ with total measure $>\eta$, then 
$J$ has cardinality $\ge e^{\delta N}$. Let $0<\epsilon<\delta$, and let $C_{\epsilon}^N$
be the subset of all $x\in X$ such that $|-\log\frac{1}{N}\mu([x]_N)-h_{\mu_x^\E}(T,\P)|>\epsilon$. 
From Theorem \ref{smb} it follows that $\mu(C_{\epsilon}^N)\to 0$ as $N\to \infty$.   
Thus for $N$ large enough $\mu(B\setminus C_{\epsilon}^N)>\eta$. Take $N$ large enough 
so that both $\mu(B\setminus C_{\epsilon}^N)>\eta$, and $N>N_0$. 
Let $J_N$ be the collection of all the partition elements of $\P_N$ that intersect $B\setminus C_{\epsilon}^N$
non-trivially. Then $J_N$ has total measure $>\eta$ and so $|J_N|\ge e^{\delta N}$.
In other words there are at least $e^{\delta N}$ partition elements of $\P_N$ of the form $[x]_N$ 
for some $x\in B\setminus C_{\epsilon}^N$. But for any such partition element we have 
$|-\log\frac{1}{N}\mu([x]_N)|\le\epsilon$, and so $\mu([x]_N)\ge e^{-\epsilon N}$. Thus for any
$N$ large enough we found a collection $J_N$ of partition elements of $\P_N$ of cardinality  
at least $e^{\delta N}$ and each of the elements of $J_N$ has measure $\ge e^{-\epsilon N}$.
Thus $J_N$ has total measure $\ge e^{(\delta-\epsilon)N}$. Taking $N\to\infty$ we have that $\mu(X)=\infty$
which is a contradiction.  
\end{proof}

\section{Diophantine approximation}\label{sec.app2}
Fix a positive integer $m$. 
For any $(\alpha_1,\dots,\alpha_m)=\alpha\in\R^m$ let $\norm{\alpha}=\max_i\norm{\alpha_i}$, where 
$\norm{\alpha_i}$ is the minimal distance to an integer. 
In this appendix we prove the following result. 
\begin{lem}\label{applem}
For any $m\in\N$ and $\epsilon>0$, there exist constants $C,C'>0$ so that for every 
$\alpha\in\R^m$ and $N\in\N$ there exist $q\in\N$ such that $C'N\le q\le C N$ and 
$\norm{q\alpha}<\epsilon$. 
\end{lem}

For a positive integer $N$ we let 
$[N]:=\{n\in\Z\mid 1\le n\le N\}$. Let $\T^m$ be the $m$ dimensional torus $\R^m/\Z^m$,
$dx$ the usual probability measure on it, and $d(x,y)$ the usual distance, 
induced from $\R^m$. 
Given $N>0$ and $\delta>0$, a finite sequence $(a_n)_{n\in[N]}\subset \R^m/\Z^m$ 
is said to be \emph{$\delta$-equidistributed} if we have 
$$
|\frac{1}{N}F(a_n)-\int_{\R^m/\Z^m}F(x)dx|\le\delta\norm{F}_{\text{Lip}},
$$
for all Lipschitz functions $F:\R^m/\Z^m\to \C$, where 
$$
\norm{F}_{\text{Lip}}:=\norm{F}_\infty+\sup_{x,y\in\R^m/\Z^m,x\ne y}\frac{|F(x)-F(y)|}{d(x,y)}.
$$
For the proof of Lemma \ref{applem} we use the following version of Kronecker's theorem.
\begin{thm}[\bf{\cite[Proposition~3.1]{Tao}}]\label{tao}
For any $0<\delta\le\frac{1}{2}$ and $\alpha\in\R^m/\Z^m$ if the sequence $(n\alpha)_{n\in[N]}$ is not
$\delta$-equidistributed then there are $k_1,\dots,k_m\in\Z^m$ with $|k_i|=O_{\delta,m}(1)$
such that $\norm{\alpha\cdot k}=O_{\delta,m}(\frac{1}{N})$.
\end{thm}
\begin{proof}[Proof of Lemma \ref{applem}]
The proof is by induction on $m$. The case where $m=1$ follows immediately from 
Dirichlet's approximation theorem. Assume $m>1$, and let $\epsilon>0$. 
There exists $l=l(\epsilon,m)\in\N$ 
and $0<\delta<\frac{1}{2}$ such that for any $N\in\N$, if $(n\alpha)_{n\in[lN]}$ 
is $\delta$-equidistributed then there exists $C_1N\le q\le C_2N$ with $\norm{q\alpha}<\epsilon$, for 
some $C_1,C_2$ depending only on $\epsilon$ and $m$.  
If $(n\alpha)_{n\in[lN]}$ is not $\delta$-equidistributed then by Theorem \ref{tao} there are
$k_1,\dots,k_m$ (not all of them are zero), 
$|k_i|=O_{\epsilon,m}(1)$ such that $\norm{\sum k_i\alpha_i}\le D\frac{1}{lN}$,
for some $D=D(\epsilon,m)$. Assume without loss of generality that $k_m\ne 0$.  
By induction we have $C'(\epsilon',m-1)N\le q\le C(\epsilon',m-1) N$ such that 
$\norm{q\alpha_i}\le \epsilon'$, for any $\epsilon'>0$. Since the $k_i$ 
are bounded in terms of $m,\epsilon$ we can choose $\epsilon'=\epsilon'(\epsilon, m)$
and $l=l(\epsilon,m)$ so that $\norm{q\alpha_m}<\epsilon$ and so we are done.   
\end{proof}

{}

\begin{thebibliography}{}
\bibitem{BrL}
J. Bourgain, E. Lindenstrauss, 
{\it Entropy of quantum limits},
Comm. Math. Phys. 233 (2003), 153–171.

\bibitem{BL} 
S. Brooks, E. Lindenstrauss,
{\it Joint quasimodes, positive entropy, and quantum unique ergodicity},  
Invent. Math. 198 (2014), 219-–259.

\bibitem{Col}
Y. Colin de Verdiere, 
{\it Ergodicite et fonctions propres du laplacien}, 
Comm. Math. Phys., 102(3):497–502, 1985.

\bibitem{EAL}
M. Einsiedler, A. Katok, E. Lindenstrauss,
{\it Invariant measures and the set of exceptions to Littlewood's conjecture},
Ann. of Math. 164 (2006), 513–-560. 

\bibitem{EK}
M. Einsiedler , A. Katok,
{\it Invariant measures on $G/\Gamma$ for split simple Lie groups $G$}
Comm. Pure Appl. Math. 56 (2003), 1184–1221.

\bibitem{Lin} 
M. Einsiedler, E. Lindenstrauss, 
{\it Diagonal actions on locally homogeneous spaces}, Homogeneous flows, moduli spaces and arithmetic, 155–241, Clay Math. Proc., 10, Amer. Math. Soc., Providence, RI, 2010.

\bibitem{EEW}
M. Einsiedler, E. Lindenstrauss, T. Ward, 
{\it Entropy in ergodic theory and homogeneous dynamics}, 
work in progress, https://tbward0.wixsite.com/books/entropy

\bibitem{Shni}
A. I. Shnirel’man, 
{\it Ergodic properties of eigenfunctions}, 
Uspekhi Mat. Nauk. 180 (1974), 181–-182.


\bibitem{Sound}
K. Soundararajan, 
{\it Quantum unique ergodicity for $\SL_2(\Z)\backslash \mathbb{H}$},
Ann. of Math. 172 (2010), 1529-–1538.
 
\bibitem{Tao} 
B. Green B, T. Tao,
{\it The quantitative behaviour of polynomial orbits on nilmanifolds},
Ann. of Math. 175 (2012), 465–-540.

\bibitem{Lin1}
E. Lindenstrauss, 
{\it Invariant measures and arithmetic quantum unique ergodicity},  
Ann. of Math. 163 (2006), 165-–219.

\bibitem{M}
I. G. Macdonald, 
{\it Spherical functions on a group of p-adic type},  
Publications of the Ramanujan Institute, No. 2. Ramanujan Institute, Centre for Advanced Study in Mathematics, University of Madras, Madras, 1971.

\bibitem{Rein}
I. Reiner,
{\it Maximal orders}, 
Corrected reprint of the 1975 original. With a foreword by M. J. Taylor. London Mathematical Society Monographs. New Series, 28. The Clarendon Press, Oxford University Press, Oxford, 2003. xiv+395 pp. ISBN: 0-19-852673-3 16H05 (11R54 16K20)

\bibitem{RS} 
Z. Rudnick Z, P. Sarnak,
{\it The behaviour of eigenstates of arithmetic hyperbolic manifolds} 
Comm. Math. Phys.161 (1994), 195-–213.

\bibitem{Sa} 
I. Satake, {\it Theory of spherical functions on reductive algebraic groups over $p$-adic fields}, Inst. Hautes Études Sci. Publ. Math.  (1963), 5--69.

\bibitem{Sil}
L. Silberman,
{\it Quantum unique ergodicity on locally symmetric spaces: the degenerate lift},
Canad. Math. Bull. 58 (2015), no. 3, 632–650. 


\bibitem{SV}
L. Silberman, A. Venkatesh, 
{\it Entropy bounds and quantum unique ergodicity for Hecke eigenfunctions on division algebras},  
Probabilistic methods in geometry, topology and spectral theory, 171–197, 
Contemp. Math., 739, Amer. Math. Soc., Providence, RI, 2019. 

\bibitem{SV2} 
L. Silberman, A. Venkatesh A, 
{\it On quantum unique ergodicity for locally symmetric spaces}, 
Geom. Funct. Anal. 17 (2007), 960-–998.

\bibitem{Zel}
S. Zelditch,
{\it Pseudodifferential analysis on hyperbolic surfaces},
J. Funct. Anal., 68(1):72–105, 1986.

\end{thebibliography}
\end{document}